\documentclass[bj,pdftex]{imsart}

\usepackage[T1]{fontenc}
\usepackage[latin1]{inputenc}
\usepackage{amsmath}
\usepackage{amssymb}
\usepackage{amsthm}
\usepackage{epsfig}
\usepackage{times}

\arxiv{arXiv:1407.0335}

\startlocaldefs

\theoremstyle{plain}

\theoremstyle{remark}

\endlocaldefs

\usepackage{amsthm}

\usepackage[usenames,dvipsnames]{color}
\definecolor{Dark}{gray}{.65}

\newcommand{\changes}[1]{\color{black} {#1} \color{black}}

\theoremstyle{plain}									
\newtheorem{theorem}{Theorem}[section]

\newtheorem{lemm}{Lemma}[section]

\theoremstyle{remark}
\newtheorem{remark}{Remark}[section]
\numberwithin{equation}{section}

\newcounter{example}[section]

\def\R{\mathbb{R}}

\newcommand{\E}{\textnormal{E}}

\def\N{\mathbb{N}}

\newcommand{\epsi}{\varepsilon}

\newcommand{\F}{\ensuremath{\mathcal{F}}}

\renewcommand{\v}{\mathbf{v}}

\DeclareMathAlphabet{\mathpzc}{OT1}{pzc}{m}{it}
\newcommand{\fha}{\hat{f}}

\newcommand{\I}{\mathbb{I}}

\usepackage{float}
\renewcommand{\H}{\mathbb{H}}
\newcommand{\un}[1]{\| {#1} \|_{\mathbb{H}_1}} 
\newcommand{\deux}[1]{\| {#1} \|_{\mathbb{H}_2}}

\makeatother
\usepackage[numbers,sort]{natbib}

\usepackage{enumerate}

\begin{document}

\begin{frontmatter}
\title{A general approach to posterior contraction in nonparametric inverse problems}
\runtitle{Posterior contraction in inverse problems}

\begin{aug}
\author{\fnms{Bartek} \snm{Knapik}\thanksref{a,e1}\ead[label=e1,mark]{b.t.knapik@vu.nl}}
\and
\author{\fnms{Jean-Bernard} \snm{Salomond}\thanksref{b,e2}\ead[label=e2,mark]{jean-bernard.salomond@u-pec.fr}}

\address[a]{Department of Mathematics, Vrije Universiteit Amsterdam, De Boelelaan 1081, 1081 HV Amsterdam, The Netherlands.
\printead{e1}}

\address[b]{Universit\'e Paris-Est, Laboratoire d'Analyse et de Math\'ematiques Appliqu\'ees (UMR 8050), UPEM, UPEC, CNRS, F-94010, Cr\'eteil, France.
\printead{e2}}

\runauthor{Bartek Knapik and Jean-Bernard~Salomond}

\affiliation{Vrije Universiteit Amsterdam, Universit\'e Paris-Est Cr\'eteil}

\end{aug}

\begin{abstract}
In this paper we propose a general method to derive an upper bound 
for the contraction rate of the posterior distribution for nonparametric 
inverse problems. We present a general theorem that allows us to derive 
contraction rates for the parameter of interest from contraction rates 
of the related direct problem of estimating transformed parameter of interest. 
An interesting aspect of this approach 
is that it allows us to derive contraction rates for priors that are 
not related to the singular value decomposition of the operator. 
We apply our result to several examples of linear inverse problems, 
both in the white noise sequence model and the nonparametric regression model, 
using priors based on 
the singular value decomposition of the operator, location-mixture 
priors and splines prior, and recover minimax adaptive contraction rates. 
\end{abstract}


\end{frontmatter}

%
%
\section{Introduction}

Statistical approaches to inverse problems have
been initiated in the 1960's and since then many estimation methods have been
developed.  Inverse problems arise naturally when one observes the object 
of interest only indirectly. Mathematically speaking, this
phenomenon is easily modeled by the introduction of an operator $K$ modifying 
the object of interest $f$, such that the observation at hand comes from the model 
\begin{equation} 
	Y^n \sim P_{Kf}^n, 
\label{eq:intro:model} 
\end{equation} 
where $f$ is assumed to belong to a parameter space $\mathcal{F}$, 
and $n \to \infty$ reflects the increasing amount of information in the observation. In many
applications the operator $K$ is assumed to be injective. However, in the most
interesting cases its inverse is not continuous, thus the parameter of
interest $f$ cannot be reconstructed by a simple inversion of the operator. 
Such problems are said to be \emph{ill-posed}.  %
Several methods dealing with the discontinuity of the inverse operator 
have been proposed in the literature. The most famous one is to
conduct the inference while imposing some regularity constraints on the 
parameter of interest $f$. These so-called
regularization methods have been widely studied in the literature both from a 
theoretical and applied perspective, see \citep{cavalier2008inverse,engl1996regularization} for reviews.

A Bayesian approach to inverse problems is therefore particularly interesting, as it is well
known that putting a prior distribution on the functional parameter yields a natural regularization. This
property of the Bayesian approach is particularly interesting for model choice, but it has proved 
also useful in many estimation procedures, as shown in \citep{Rousseau2011overfitted} in the 
case of overfitted mixtures models,  in \citep{castillo2013bayesian} in the case of nonparametric 
models where regularization is necessary, in \citep{salomond2013concentration} in the semiparametric problem
of estimating a monotone density at the boundaries of its support, or in \citep{KVZ11} in the white noise setting.  

In this paper we study the behaviour of the posterior distribution when 
the amount of information goes to infinity (e.g. when the number of data points $n$ goes to infinity 
or when the level of the noise goes to $0$) under the
frequentist assumption that the data $Y^n$ are generated from model (\ref{eq:intro:model}) for some
true unknown parameter $f_0$.
{Asymptotic properties of the 
posterior distribution in nonparametric models have been studied for many years. 
Some first results about consistency of Bayes procedures date back to Schwartz \citep{Schwartz65}. 
Her ideas were further refined and extended in an unpublished work of Barron \citep{Barron88}, 
and can be also found in other works, e.g., \citep{Barron99}, \citep{GGR99}. The next 
natural step is to consider the rate at which the neighborhoods of the truth can shrink, 
yet still capture most of the posterior mass. In other words, the interest lies in finding an upper bound for 
the rate at which the posterior concentrates around $f_0$. This is also the main 
focus of this paper. The aforementioned consistency results served as a starting point 
for two seminal papers on rates of convergence of posterior distributions by Ghosal et al.\ 
\citep{ghosal2000convergence} and Shen and Wasserman \citep{Shen2001}.}
Understanding of the whole posterior distribution is necessary for uncertainty quantification, 
see a recent paper by Szab\'o et al.\ \citep{szabocredible} for an overview, but is also 
directly related to asymptotic properties of Bayes point estimators, 
{see, e.g., Theorem~2.5 in \citep{ghosal2000convergence}}. 
In Bayesian nonparametrics it is important 
to understand the impact of the prior distribution on the posterior. 
In particular, some aspects of the prior may be inherited by the posterior when the amount of information 
grows to infinity and may thus be highly influential for the quality and speed of recovery. 

Asymptotic properties of the Bayesian approach to nonparametric linear inverse problems have recently received a growing interest. 
Knapik et al.\ \citep{KVZ11}, Agapiou et al.\ \citep{Agapiou12}, and Florens and Simoni \citep{florens2012regularized} 
were the first to study posterior contraction rates under conjugate prior in the so-called mildly ill-posed setting 
(in the terminology of \citep{cavalier2008inverse}). 
These were followed by two papers by Knapik et al.\ \citep{KVZ13} and Agapiou et al.\ \citep{Agapiou14}, 
studying Bayesian recovery of the initial condition for heat equation and related extremely and 
severely ill-posed inverse problems. One type of priors studied in \citep{KVZ13} leads to a rate-adaptive Bayesian procedure. 
The paper by Ray \citep{ray2013bayesian} was the first study of the posterior contraction rates 
in the non-conjugate sequence setting. 
Considering non-conjugate prior is particularly interesting as it allows some additional 
flexibility of the model. However, the approach presented in \citep{ray2013bayesian} is only valid for priors 
that are closely linked to the \emph{singular value decomposition} (SVD) of the operator. 
Moreover, in \citep{ray2013bayesian} several rate-adaptive priors were considered, 
both in the mildly and severely ill-posed setting. 
It should be noted, however, that some of the bounds on contraction rates in the severely ill-posed 
setting obtained in that paper are not optimal and do not agree with 
the bounds found in \citep{KVZ13} or \citep{Agapiou14}, 
probably due to proof techniques.  
Similar adaptive results, in the conjugate mildly ill-posed setting, using 
empirical and hierarchical Bayes approach were obtained in \citep{KSzVZ12}. 

There is a rich literature on the problem of deriving posterior contraction 
rate in the direct problem setting, i.e. estimating $Kf$ in \eqref{eq:intro:model}. Since the seminal papers of Ghosal et al.\ 
\citep{ghosal2000convergence} and Shen and Wasserman \citep{Shen2001}, general conditions on the
prior distribution for which the posterior contracts at a certain rate have
been derived in various cases. In particular, Ghosal and van~der~Vaart in \citep{GhosalVdv07} 
give a number of conditions for non independent and identically distributed data. However, 
such results cannot be applied directly to ill-posed inverse problems, 
and to the authors' best knowledge, no analogous results exist in the inverse 
problem literature. In this paper we propose a unified general approach 
to posterior contraction in nonparametric inverse problems, and illustrate it for specific linear inverse problems.

To understand why the existing general posterior contraction results are not suited 
for nonparametric inverse problems consider an abstract setting 
in which the parameter space $\mathcal{F}$ is an arbitrary 
metrizable topological vector space and let $K$ be a continuous injective mapping 
$K: \mathcal{F} \ni f \mapsto Kf \in K\mathcal{F}$. Let $d$ and $d_K$ 
denote some metrics or semi-metrics on $\F$ and $K\F$, respectively. 
Any prior $\Pi$ on $f$ imposes a prior on $Kf$ through the continuous mapping $K$. 
Recall that the true parameter of interest $f_0$ belongs to $\F$. 
General posterior contraction results (e.g., in \citep{ghosal2000convergence} 
or \citep{GhosalVdv07}) rely on several natural metrics related to 
the model (\ref{eq:intro:model}) and therefore control the distance 
between $Kf_0$ and $Kf$ in the $d_K$ metric. On the other hand, our 
interest lies in the recovery of $f_0$, and therefore 
the control of the distance between $f_0$ and $f$ in the $d$ metric 
is desirable. Since the operator $K$ does not have a continuous inverse 
and the problem is ill-posed, even if $d_K(Kf, Kf_0)$ is small, 
the distance $d(f, f_0)$ between $f$ and the true $f_0$ can be arbitrarily large. 
In other words, there is no equivalence between the metrics $d$ and $d_K$ 
and therefore the existing theory of posterior contraction does not 
allow obtaining bounds on posterior contraction rates for the 
recovery of $f_0$. 

Even if the problem is ill-posed, there exist subsets $\mathcal{S}_n$ of $\F$ 
such that the inverse of the operator $K$ restricted to $K\mathcal{S}_n$ is 
continuous. We can thus easily 
derive posterior contraction rate for $f \in \mathcal{S}_n$ from 
posterior contraction rate for $Kf$ by inverting the operator $K$. 
For suitably chosen priors, the sets $\mathcal{S}_n$ will capture most 
of the posterior mass, and we can thus extend the contraction result to the whole 
parameter space $\F$. 
{The sets $\mathcal{S}_n$, thought of as 
sieves approximating the parameter space $\F$, have already 
been considered in \citep{ghosal2000convergence} allowing 
some additional flexibility and are often incorporated 
in results on posterior contraction for various models. 
However, their principal role was not to enable 
the change of metrics, but rather alleviate the usual entropy 
condition.} 
%
{
In our approach we first assume the 
existence of a contraction result for the so-called direct problem 
(that is the recovery of $Kf$) that can be derived using general 
posterior contraction literature. Next, we choose a sequence of 
subsets $\mathcal{S}_n$ in such a way that the inversion of the operator 
$K$ on $K\mathcal{S}_n$, so also the change of metrics, can be controlled and at the same this sets 
are big enough 
(in terms of the posterior mass).} The latter condition 
can be verified by imposing additional sufficient 
conditions on the prior (on $f$). We are then able to show that 
the posterior distribution for the parameter of interest $f$ contracts 
at a given rate. 

The rest of the paper is organized as follows: we present 
the main result in Section \ref{sec:general:th} 
{and 
discuss how it relates to other results using the concept of sieves 
to control contraction in other metrics}. 
We then apply our result in various settings. We first consider 
the white noise sequence model in Section~\ref{sec:sequence:model}, 
where we present a general construction of the sets $\mathcal{S}_n$, 
and recover many of the existing results with much less effort. 
We also observe an interesting interplay between optimality of Bayesian 
procedures for estimating $f$ and $Kf$.
In Section \ref{sec:regression} we apply our method in the 
nonparametric inverse regression setting, considering 
new families of priors that need not be related to the SVD, and leading 
to optimal Bayesian procedures. 
Proofs of Sections~\ref{sec:general:th}--\ref{sec:regression} are placed in Section~\ref{sec:proofs}.
We conclude the paper with a discussion in Section~\ref{sec:discussion}.

For two sequences $(a_n)$ and $(b_n)$ of numbers,  $a_n \asymp b_n$ means that
$|a_n/b_n|$ is bounded away from zero and infinity as $n \to \infty$, $a_n
\lesssim b_n$ means that $a_n/b_n$ is bounded, $a_n \sim b_n$ means  that
$a_n/b_n \to 1$ as $n \to \infty$, and $a_n \ll b_n$ means that  $a_n / b_n
\to 0$ as $n \to \infty$. For two real numbers $a$ and $b$,  we denote by $a
\vee b$ their maximum, and by $a \wedge b$ their minimum. For a sequence of
random variables $X^n = (X_1, \dots , X_n) \sim P^n_f$ and any measurable
function $\psi$ with respect to $P_f^n$, we denote by $\E_f \psi$ the
expectation of $\psi(X^n)$ with respect to $P^n_f$ and when $f = f_0$ we will write $\E_0$ instead of 
$\E_{f_0}$. 


\section{General theorem}
\label{sec:general:th}

Assume that the observations $Y^{n}$ come from model (\ref{eq:intro:model}) 
and that $P^n_{Kf}$ admit densities
$p^{n}_{Kf}$ relative to a $\sigma$-finite measure $\mu^{n}$.
To avoid complicated notations, we drop the superscript $n$ in
the rest of the paper. Let $\mathcal{F}$ and $K\mathcal{F}$
be metric spaces, and let $d$ and $d_K$ denote metrics on
both spaces, respectively.

In this section we present the main result of this paper which
gives an upper bound on the posterior contraction rate under
some general conditions on the prior.
We call the estimation of $Kf$ given the observations $Y$
the \emph{direct problem}, and the estimation $f$ given $Y$ the
\emph{inverse problem}. The main idea is to control the change 
of metrics $d_K$ and $d$. If the posterior distribution 
concentrates around $Kf_0$ for the metric $d_K$ at a certain 
rate in the direct problem, applying the change of metrics will 
give us an upper bound on the posterior contraction rate 
for the metric $d$ in the inverse problem. However, since the 
operator $K$ does not posses a continuous inverse, 
the change of metrics cannot be controlled 
over the whole space $K\mathcal{F}$. A way to circumvent this 
issue is to only focus on a sequence of sets of high posterior 
mass for which the change of metric is feasible. More precisely,  
for a set $\mathcal{S}\subset \mathcal{F}$, $f_0 \in \mathcal{F}$ 
and a fixed $\delta > 0$ 
we call the quantity
\begin{equation}
    \omega(\mathcal{S},f_0, d, d_K,\delta)
        := \sup \bigl\{ d(f, f_0): f \in \mathcal{S}, d_K(Kf, Kf_0) \leq \delta \bigr\}
\label{eq:def:modulus}
\end{equation}
the \emph{modulus of continuity}. 
We note that in this definition we do not assume $f_0 \in \mathcal{S}$. 
This is thus a local version of the modulus of continuity considered in 
\citep{donoho1991} or \citep{hoffmann2013adaptive}. 
On the one hand, the sets $\mathcal{S}_n$ need to be 
big enough to capture most of the posterior mass. On the other hand, one 
has to be able to control the distance between the elements of 
$\mathcal{S}_n$ and $f_0$, given the distance between $Kf$ and 
$Kf_0$ is small. Since the operator $K$ is unbounded, this suggests 
that the sets $\mathcal{S}_n$ cannot be too big. 

\begin{theorem}\label{thm:main}
\label{th:main}
Let $\epsilon_n \to 0$  and let $\Pi$ the prior distribution on $f$ be such that
\begin{equation}
	\E_0\Pi\bigl(\mathcal{S}_n^c\, | \, Y^n\bigr) \to 0,
\label{eq:postmass}
\end{equation}
for some sequence of sets $(\mathcal{S}_n)$, $\mathcal{S}_n \subset \mathcal{F}$, 
and for any positive sequence $M_n$
\begin{equation}
    \E_0\Pi\bigl(f: d_K(Kf, Kf_0)\geq M_n\epsilon_n \, | \, Y^n\bigr) \to 0.
\label{eq:mainth:condition:contraction:direct}
\end{equation}
Then
\[
    \E_0\Pi\bigl(f: d(f,f_0) \geq \omega(\mathcal{S}_n,f_0,d, d_K,M_n\epsilon_n)\, | \,  Y^n\bigr) \to 0.
\]
\end{theorem}

The proof is elementary and can be found in Section \ref{sec:proof:main}. 

The idea behind Theorem~\ref{th:main} is simple 
{and was used to change metrics also in direct problems. 
For instance Castillo and van der Vaart \citep{Castillo12} 
considered the multivariate normal mean
model in the situation that the mean vector is sparse.  
They use the fact that the posterior concentrates along 
certain subspaces on which it is easy to control an 
$\ell_q$-like metric with the standard Euclidean 
metric for $q < 2$.
Hoffmann et al.\ \citep{hoffmann2013adaptive} also use concentration 
of the posterior on specific sets to control 
the $L_\infty$ metric with the $L_2$ metric in the white noise model.

Castillo et al.\ \citep{Castillo15} extended the ideas of \citep{Castillo12} 
to the sparse linear regression model, in which the recovery of the 
parameter of the model is an inverse problem. 
Similar reasoning was also used in \citep{vollmer2013posterior} 
to study posterior contraction in the special case of Gaussian 
elliptic inverse problem, and in \citep{donnet2014posterior} to investigate 
asymptotic properties of empirical Bayes procedures for density deconvolution. 
However, these papers consider specific inverse problems only, 
whereas Theorem~\ref{th:main} allows 
deriving contraction rates for a wide variety of inverse problem 
models for which the prior is not necessarily 
related to the spectral decomposition of the operator $K$, e.g., 
when the operator does not admit singular value decomposition, as 
in Section~\ref{sec:splines}.
}

The interpretation of the theorem is the following: given a properly 
chosen sequence of sets $\mathcal{S}_n$, the rate of posterior contraction $M_n\epsilon_n$ 
in the direct problem restricted to the given sequence can be translated 
to the rate of posterior contraction in the inverse setting. 
Note that the sequence $M_n$ is often chosen to grow to infinity 
as slowly as needed (see, e.g., in \citep{ghosal2000convergence} 
or \citep{GhosalVdv07}), making $\epsilon_n$ the effective rate 
of posterior contraction. 
{Also, in both contraction results of Section~4, the sequence $M_n$ 
need not be diverging and is chosen to be constant.} 
Since the operator $K$ is injective and continuous, any prior $\Pi$ on $f$ 
induces a prior on $Kf$, and the general posterior contraction results 
can be applied to obtain the rate of contraction in the direct problem 
of estimating $Kf$.

Next, the choice of $\mathcal{S}_n$ is crucial as it is the principal 
component in the control of the change of metric. In particular, the 
contraction rate $M_n\epsilon_n$ for the direct problem may not be 
optimal, and still lead to an optimal contraction rate 
$\omega(\mathcal{S}_n,f_0,d, d_K,M_n\epsilon_n)$ for the inverse 
problem with a well-suited choice of $\mathcal{S}_n$. 
\changes{As shown in Section 3.3, it is possible in some cases to obtain optimal recovery of $f$ without having optimal recovery of $Kf$. In this example, we can choose $S_n$ small enough so that the change of metrics can be control very precisely. This widens the possible choice of priors leading to optimal contraction rates and shows that the change of metric is the crucial part here. However, in most cases, the priors considered in this paper lead to optimal recovery for both $f$ and for $Kf$.}

To control the posterior mass of the sets $\mathcal{S}_n$ we can usually 
alter the proofs of contraction results for the direct problems. Here 
we present a standard argument leading to (\ref{eq:postmass}). 
Define the usual Kullback--Leibler neighborhoods by
\begin{equation}
\begin{split}
    B_n(Kf_0, \epsilon) = \Bigl\{f\in \mathcal{F}:{}
    -\int p_{Kf_0}\log\frac{p_{Kf}}{p_{Kf_0}}\, d\mu &\leq n\epsilon^2,\\
    \int p_{Kf_0}\Bigl(\log\frac{p_{Kf}}{p_{Kf_0}}\Bigr)^2\, d\mu &\leq n\epsilon^2,
    \Bigr\},
\end{split}
\label{eq:main:de:KLneighbour}
\end{equation}
The following lemma adapted from \citep{GhosalVdv07} gives general 
conditions on the prior such that (\ref{eq:postmass}) is satisfied.  
\begin{lemm}[Lemma~1 in \citep{GhosalVdv07}]
Let $\epsilon_n \to 0$ and let $(\mathcal{S}_n)$ be a sequence of sets 
$\mathcal{S}_n \subset \mathcal{F}$. If $\Pi$ is the prior distribution on $f$ satisfying
\[
    \frac{\Pi(\mathcal{S}_n^c)}{\Pi(B_n(Kf_0, \epsilon_n))} \lesssim \exp(-2n\epsilon_n^2),
\]
then
\[
    \E_0\Pi\bigl(\mathcal{S}_n^c\, | \, Y^n\bigr) \to 0.
\]
\label{lem:main}
\end{lemm}
For clarity of presentation the results in this section are stated 
for a fixed $f_0$, but we note that they are easily extended to 
uniform results over certain sets, i.e., balls of fixed radius and regularity, 
or union of balls of fixed radii over compact range of regularity parameter (see 
results of Section~\ref{sec:sequence:model}).


\section{Sequence white noise model}
\label{sec:sequence:model}

Our first examples are based on the well-studied infinite-dimensional 
normal mean model. In the Bayesian context the problem of direct 
estimation of infinitely many means has been studied, among others, in 
\citep{Zhao,  Shen2001, Belitser, GhosalVdv07}.  

We consider the white noise setting, where we observe an infinite 
sequence $Y^n=(Y_1, Y_2, \ldots)$ satisfying 
\begin{equation}
    Y_i = \kappa_if_i + \frac{1}{\sqrt{n}}Z_i,
\label{eq:modelWN}
\end{equation}
where $Z_1, Z_2, \ldots$ are independent standard normal random variables, 
$f = (f_1, f_2, \ldots) \in \ell_2$ is the infinite-dimensional parameter of 
interest and $(\kappa_i)$ is a known sequence that may converge to $0$ as 
$i \to \infty$. If this is the case (so 
when the operator $K$ does not possess a continuous inverse) the modulus of continuity defined 
in (\ref{eq:def:modulus}) is infinite when $\mathcal{S} = \mathcal{F}$. 

Even though this model is rather abstract, it is mathematically tractable 
and it enables rigorous results and proofs. Moreover, it can be seen as 
an idealized version of other statistical models through equivalence 
results 
{see, e.g., \citep{Meister11, Nussbaum96, Brown96}}. 
Both white noise examples of inverse problems 
presented in this section have already 
been studied in the Bayesian literature. 
We present them here for several reasons. 
First, the direct version of the normal mean model attracted a lot of attention 
in the Bayesian literature, e.g. providing contraction results for estimation of $Kf$ 
in the mildly ill-posed setting. Therefore, we choose this example to illustrate 
how Theorem~\ref{th:main} works in practice. In particular, it allows us to make 
it clear how one could construct a sequence of sets $\mathcal{S}_n$.
In the severely ill-posed case we study truncated (or sieve) priors leading 
to optimal recovery of the parameter of interest. Our results improve 
the findings of \citep{KVZ13} and \citep{Agapiou14}. 
In addition, we can show that optimal contraction 
for $f$ does not necessarily require optimal recovery of $Kf$.


\subsection{Computation of a modulus}
\label{sec:modulus}

In this section we first present an example of the sequence 
of sets $\mathcal{S}_n$, and later present how the modulus of 
continuity for this sequence can be computed in a standard 
inverse problem setting. We now suppose that $\mathcal{F}$ and 
$K\mathcal{F}$ are separable Hilbert spaces, 
denoted $(\H_1, \|\cdot\|_{\H_1})$ and $(\H_2,\|\cdot\|_{\H_2})$ 
respectively. We note that the sets $\mathcal{S}_n$ resemble 
the sets $\mathcal{P}_n$ considered in \citep{ray2013bayesian}.

As already noted, the operator $K$ restricted to certain subsets of 
the domain $\H_1$ might have a finite modulus of continuity 
defined in (\ref{eq:def:modulus}). Clearly, one wants to construct 
a sequence of sets $\mathcal{S}_n$ that in a certain sense approaches 
the full domain $\H_1$. This is understood in terms of the 
remaining prior mass condition in Theorem~\ref{thm:main}. Moreover, since 
we do not require $f_0$ to be in $\mathcal{S}_n$, we need to be able to 
control the distance between $f_0$ and $\mathcal{S}_n$.  

A natural guess is to consider finite-dimensional projections of $\H_1$. 
In this section we go beyond this concept. To get some intuition, 
consider the Fourier basis of $\H_1$. The ill-posedness can be 
then viewed as too big an amplification of the high frequencies 
through the inverse of the operator $K$. Therefore, one wants 
to control the higher frequencies in the signal, and thus in 
the parameter $f$. 

Since $\H_1$ is a separable Hilbert space, there exist an 
orthonormal basis $(e_i)$ and each element $f \in \H_1$ 
can be viewed as an element of $\ell_2$ and
\[
    \un{f} = \sum_{i=1}^\infty f_i^2.
\]
For given sequences of positive numbers $k_n \to \infty$ and $\rho_n \to 0$, and a constant $c \geq 0$ we define
\begin{equation}
\label{eq:Sn}
    \mathcal{S}_n := 
       \Bigl\{f\in\ell_2: \sum_{i>k_n} f_i^2 \leq c\rho_n^2\Bigr\}.
\end{equation}

If the operator $K$ is compact, then the spectral decomposition of the 
self-adjoint operator $K^TK:\H_1 \to \H_1$ provides 
a convenient orthonormal basis. In the compact case the operator 
$K^TK$ possesses countably many positive eigenvalues $\kappa_i^2$ 
and there is a corresponding orthonormal basis $(e_i)$ of $\H_1$ 
of eigenfunctions, and the sequence $(\tilde e_i)$ defined by 
$Ke_i = \kappa_i\tilde e_i$ forms an orthonormal conjugate basis of the 
range of $K$ in $\H_2$. Therefore, both $f$ and $Kf$ can be associated 
with sequences in $\ell_2$. Since the problem is ill-posed when $\kappa_i \to 0$, 
we can assume without loss of generality that the sequence $\kappa_i$ is decreasing.

Let $k_n$, $\rho_n$, and $c$ in the definition of $\mathcal{S}_n$ be fixed. Then for any 
$g \in \mathcal{S}_n$
\begin{align*}
    \un{g}^2 &= \sum_{i=1}^\infty g_i^2 = \sum_{i \leq k_n}g_i^2  + \sum_{i > k_n}g_i^2 \\
       &\leq  \sum_{i \leq k_n}g_i^2 +c\rho_n^2 = \sum_{i \leq k_n}\kappa_i^{-2}\kappa_i^2 g_i^2 +c\rho_n^2\\
       &\leq  \kappa_{k_n}^{-2}  \sum_{i \leq k_n}\kappa_i^2 g_i^2 +c\rho_n^2 \leq \kappa_{k_n}^{-2}\deux{Kg}^2 +c\rho_n^2.
\end{align*}

Let $f_{n}$ be the projection of $f_0$ on the first $k_n$ coordinates, i.e., 
$f_{n,i} = f_{0,i}$ for $i \leq k_n$ and $0$ otherwise. Moreover, we assume 
that $f_0$ belongs to some smoothness class described by a decreasing sequence 
$(s_i)$:
\[
	\|f_0\|^2_s = \sum_{i=1}^\infty s_i^{-2}f_{0,i}^2 < \infty.
\]
For instance, the usual Sobolev space of regularity $\beta$ is defined in that way with 
$s_i = i^{-\beta}$. Therefore, we have
\[
	\|f_n-f_0\|_{\H_1} \leq s_{k_n}\|f_0\|_s, \qquad 
	\|Kf_n-Kf_0\|_{\H_2} \leq s_{k_n}\kappa_{k_n}\|f_0\|_s.
\]
Using the triangle inequality twice and keeping in mind that $f-f_n \in \mathcal{S}_n$ 
we obtain
\begin{align}
	\|f-f_0\|_{\H_1} &\leq \|f-f_n\|_{\H_1} + \|f_n-f_0\|_{\H_1} \nonumber \\
		& \leq \kappa_{k_n}^{-1}\|Kf-Kf_n\|_{\H_2} + \sqrt{c}\rho_n + s_{k_n}\|f_0\|_s \nonumber\\
		& \leq \kappa_{k_n}^{-1}\bigl(\|Kf-Kf_0\|_{\H_2} + \kappa_{k_n}s_{k_n}\|f_0\|_s\bigr) 
			+ \sqrt{c}\rho_n + s_{k_n}\|f_0\|_s \nonumber\\
		& = \kappa_{k_n}^{-1}\|Kf-Kf_0\|_{\H_2} + \sqrt{c}\rho_n + 2\|f_0\|_s s_{k_n} \label{eq:seq:ineq:modulus}.
\end{align}
We then find an upper bound for the modulus of continuity with this specific choice of $\mathcal{S}_n$ is
\begin{equation}
\label{eq:mod}
    \omega(\mathcal{S}_n,f_0, \|\cdot\|_{\H_1}, \|\cdot\|_{\H_2},\delta) \lesssim \kappa_{k_n}^{-1}\delta + \rho_n + s_{k_n}.
\end{equation}



\subsection{Mildly ill-posed problems}


In this section we consider the model (\ref{eq:modelWN}), 
where $C^{-1}i^{-p} \leq \kappa_i \leq Ci^{-p}$ for some $p \geq 0$ and $C\geq 1$. 
Since the $\kappa_i$'s decay polynomially, the operator is \emph{mildly} ill-posed. 
Such problems are well studied in the frequentist literature, and 
we refer the reader to \citep{cavalier2008inverse} for a comprehensive overview. There are 
also several papers on properties of Bayes procedures for such problems. The first studies 
of posterior contraction in mildly ill-posed operators were obtained in 
\citep{KVZ11} and \citep{Agapiou12}. Later, adaptive priors leading to the optimal 
minimax rate of contraction 
{(up to slowly varying factors)} 
were studied in \citep{ray2013bayesian} and \citep{KSzVZ12}. 
Similar problem, with a different noise structure, has been studied in \cite{florens2012regularized}. 
The main purpose of this section is to show how Theorem 2.1 can be applied to such 
problems and how existing results on contraction rates for $Kf$ in the sequence setting 
can be used to obtain posterior contraction rates for $f$ 
without explicit computations 
as in aforementioned papers.

We put a product prior on $f$ of the form
\[
\Pi = \bigotimes_{i=1}^{\infty} N(0, \lambda_i),
\]
where $\lambda_i = i^{-1-2\alpha}$, for some $\alpha > 0$.
Furthermore, the true parameter $f_0$ is assumed to belong to $S^\beta$ for some
$\beta > 0$:
\begin{equation}
\label{eq:Sb}
    S^\beta = \Bigl\{f\in \ell_2: \|f\|_\beta^2 := \sum f_{i}^2i^{2\beta} < \infty\Bigr\}.
\end{equation}
Therefore, $\|Kf_0\|_{\beta+p}^2$ is finite, the prior on $f$ induces the 
prior on $Kf$ such that $(Kf)_i \sim N(0, \lambda_i\kappa_i^2)$, and one can deduce from 
the results of \citep{Zhao} and \citep{Belitser} that 
\[
    \sup_{\|Kf_0\|_{\beta+p} \leq R}\E_0\Pi\bigl(f: \|Kf-Kf_0\| \geq M_n n^{-\frac{(\alpha \wedge \beta)+p}{1+2\alpha+2p}}\bigm|Y^n\bigr)\to 0.
\]

In order to apply Theorem~\ref{thm:main} we need to construct the sequence 
of sets $\mathcal{S}_n$ and verify condition (\ref{eq:postmass}). We use the construction as in (\ref{eq:Sn}), 
and we verify the remaining posterior mass condition along the lines of 
Lemma~\ref{lem:main}. 

\begin{theorem}
Suppose the true $f_0$ belongs to $S^\beta$ for $\beta > 0$. 
Then for every $R > 0$ and $M_n \to \infty$
\[
    \sup_{\|f_0\|_\beta \leq R}
	  \E_0\Pi\bigl(f: \|f-f_0\| \geq M_nn^{-\frac{(\alpha \wedge \beta)}{1+2\alpha+2p}}\bigm|Y^n\bigr)\to 0. 
\]
\label{th:mildly}
\end{theorem}

The proof of this theorem is postponed to Section \ref{sec:proof:mildly}.  

The upper bound on the posterior contraction rate obtained 
in this theorem agrees with the ones already obtained in 
the existing literature (see, for instance, \citep{KVZ11,KSzVZ12, ray2013bayesian}). 
We note that the prior used above requires the knowledge of the true 
regularity parameter $\beta$ in order to achieve minimax optimal rate 
of recovery. Moreover, we note that the prior with $\alpha = \beta$ leads 
to optimal recovery of both $f$ and $Kf$. 

{The prior used in this section is rather simple and is 
not hierarchical, i.e., is not aimed at adaptive recovery. 
We have already pointed out 
that \citep{ray2013bayesian} and \citep{KSzVZ12} studied adaptive 
Bayesian approach to mildly ill-posed inverse problems and obtained 
optimal rates (up to logarithmic factors).
We would also like to 
point out that recent studies \changes{of adaptive approaches to 
the sequence white noise model \citep[e.g.][]{Arbel13,KSzVZ12}
 already consider its inverse version 
(i.e., allowing $\kappa_i \neq 1$).
In a recent work Belitser \cite{Belitser2016} even obtained adaptive posterior contraction rate
in a setting equivalent to the one considered here that could be used both for the estimation of $Kf$ and 
the estimation of $f$.
}
Therefore, even though one could consider the existing approaches studied in the literature to achieve adaptation (by first showing optimal 
contraction for $Kf$ and then applying Theorem~\ref{th:main} to 
prove contraction for $f$), this will not be treated here 
for the sake of simplicity (in the latter cases also to avoid rather artificial 
application of Theorem~\ref{th:main}). 
}

\subsection{Severely and extremely ill-posed problems}
\label{sec:severely}



We again consider the sequence white noise setting, where we observe an infinite 
sequence $Y^n=(Y_1, Y_2, \ldots)$ as in (\ref{eq:modelWN}) 
where $\kappa_i \asymp \exp(-\gamma i^p)$ for some $p \geq 1$ and $\gamma > 0$. 
We first consider estimation of $Kf_0$ that will be later used to obtain the rate
of contraction of the posterior around $f_0$. We put a product prior on $f$ of
the form
\[
\Pi = \bigotimes_{i=1}^{k_n} N(0, \lambda_i),
\]
where $\lambda_i = i^{-\alpha}\exp(-\xi i^p)$, for $\alpha \geq 0$, 
$\xi > 0$, and some $k_n \to \infty$. 
We choose $k_n$ solving 
$1 = n\lambda_i\exp(-2\gamma i^p) = ni^{-\alpha}\exp(-(\xi+2\gamma)i^p)$. 
Using the Lambert function $W$ one can show that
\begin{equation}
\label{eq:k_n}
    k_n = \Bigl(\frac{\alpha}{p(\xi+2\gamma)}
            W\Bigl(n^{\frac{p}{\alpha}}\frac{p(\xi+2\gamma)}{\alpha}\Bigr)\Bigr)^{1/p}
        = \Bigl(\frac{\log n}{\xi+2\gamma} + O(\log\log n)\Bigr)^{1/p},
\end{equation}
see also Lemma A.4. in \citep{KVZ13}. Note that in this case we have 
$\exp(k_n^p) = (nk_n^{-\alpha})^{1/(\xi+2\gamma)}$, so we can avoid 
exponentiating $k_n$. Therefore, we do not have to 
specify the constant in front of the $\log\log n$ term 
in the definition of $k_n$, and we may assume that $k_n$ is of the 
order $(\log n)^{1/p}$.

Note that the hyperparameters of the prior do not depend on $f_0$, 
but only on $K$, which is known. For 
$\mathcal{S}_n$ as in (\ref{eq:Sn}) 
with $k_n$ as above and $c = 0$, the prior is supported on $\mathcal{S}_n$ 
and the first condition of Theorem~\ref{thm:main} is trivially satisfied.
Regardless of the choice of $\xi$ and $\alpha$ (as long as $\alpha \geq 0$ and 
$\xi > 0$) the following theorem shows 
that the posterior contracts at
the optimal minimax rate $(\log n)^{-\beta/p}$ 
for the inverse problem of estimating $f_0$ 
(cf.\ \citep{KVZ13} or \citep{Agapiou14} and references therein), 
so the prior is rate-adaptive.

In this section we consider deterministically 
truncated Gaussian priors.
Similar priors in the extremely ill-posed setting are considered in 
\citep{ray2013bayesian}, but in this paper the truncation level
is endowed with a hyper-prior and the bound on the posterior contraction 
is suboptimal. Other papers on Bayesian approach to severely 
and extremely ill-posed inverse problems do not consider truncated priors. 
In \citep{KVZ13} the 
optimal rate is achieved for the priors with exponentially decaying 
or polynomially decaying variances (in the latter case the speed of 
decay leading to optimal rate is closely related to the regularity of 
the truth). Similar results for the priors with polynomially decaying variances 
are presented in \citep{ray2013bayesian} and \citep{Agapiou14}. 
However, in the former case the rate for undersmoothing priors is 
worse than the rate obtained in the other papers. 

\begin{theorem}
Suppose the true $f_0$ belongs to $S^\beta$ for $\beta > 0$. 
Then for every $R > 0$ and $M_n \to \infty$
\[
    \sup_{\|f_0\|_\beta \leq R}
	  \E_0\Pi\bigl(f: \|f-f_0\| \geq M_n(\log n)^{-\frac{\beta}{p}}\bigm|Y^n\bigr)\to 0.
\]
\label{th:WN:extremely}
\end{theorem}

The proof of this Theorem is postponed to Section \ref{sec:proof:severly}. 
The prior considered in this theorem might seem unnatural, since $\lambda_i$'s 
do not coincide with the type of regularity of the truth and 
the prior puts mass only on analytic functions of growing 
complexity. However, similar approaches are quite common in the Bayesian literature, 
for instance when finite mixtures models are considered. Moreover, 
this prior has also some computational advantages, since the corresponding 
posterior can be handled numerically. 

Inspection of the proof shows that the deterministic truncation 
is suboptimal for the estimation of $Kf_0$, since the resulting upper bound 
is polynomially slower than the minimax rate $n^{-1/2}(\log n)^{1/2p}$. 
It sheds light on an interesting, although counterintuitive property 
of the Bayesian approach to inverse problems: one may not need optimal 
contraction for the estimation of $Kf_0$ to get optimal contraction 
for the estimation of $f_0$. This phenomenon should be interpreted in the 
following way: since the operator $K$ regularizes the parameter $f_0$, one 
could compensate the suboptimal contraction of the posterior 
for the direct problem, by a sharper control of the deviation between 
$f$ and $f_0$ in \eqref{eq:seq:ineq:modulus} when $f$ is in $\mathcal{S}_n$. 
Indeed, when $\xi$ increases (which slows down the upper bound on the 
posterior contraction for $Kf_0$), the truncation level $k_n$ decreases. 
As a result, the sets $\mathcal{S}_n$ become smaller, so the sharper control 
of $d(f,f_0)$ is indeed possible. In the specific setting of sequence white noise 
model it might seem artificial. However, this observation could prove useful 
in more complex settings, especially because it widens the 
class of possible prior distributions giving optimal contraction rates.

\begin{remark}  
If an upper bound $\widetilde \beta$ on the regularity of the true $f_0$ 
is known, one can also take $\xi = 0$ and $\alpha \geq 1+2\widetilde \beta$ 
and the assertion of Theorem~\ref{th:WN:extremely} stays valid. In this case 
the upper bound on the posterior contraction rate for $Kf_0$ 
is logarithmically slower than the minimax rate. 
\end{remark}

\section{Regression}
\label{sec:regression}
We now consider the inverse regression model with Gaussian residuals 
\begin{equation}
	Y_i = Kf(x_i) + \sigma \epsilon_i , 
	\quad \epsilon_i \overset{iid}{\sim} \mathcal{N}(0,1), 
	\quad i = 1, \ldots, n, 
	\label{eq:regression:model}
\end{equation}
where the covariates $x_i$ are fixed in a covariate space $\mathcal{X}$. 
In the sequel, we either choose $\mathcal{X} = [0,1]$ or $\mathcal{X} = \R$. 
In the following we consider the noise level $\sigma>0$ to be known although 
one could also think of putting a prior on it and estimate it in the direct model. 
Nonparametric regression models have been studied in the literature for 
direct problems, and frequentist properties of the posterior distribution 
are well known for a wide variety of priors. In \citep{GhosalVdv07}, 
Ghosal and van~der~Vaart give general conditions on the prior such that 
the posterior contracts at a given rate. Nonparametric inverse regression models 
are also used in practice, for instance in econometrics where one considers 
instrumental variable as in \citep{Florens2012458}. However, to the authors' best 
knowledge, contraction rates for these models have only been 
considered in \citep{vollmer2013posterior}.

In this setting, a common choice for the metrics $d$ and $d_K$ are the usual $l_2$ norms
\[
	d(f,g)^2 = n^{-1} \sum_{i=1}^n (f(x_i) - g(x_i))^2 = \|f-g\|_n^2, \quad d_K(f,g) = d(Kf,Kg). 
\]
For $a\in\R^k$, $k \in \mathbb{N}^*$, and $f \in L_2$, we denote the standard Euclidean and $L_2$ norms by 
\[
	\|a\|_k = \Bigl(\sum_{i=1}^k a_i^2\Bigr)^{1/2}, \qquad
	\|f\| = \Bigl(\int f^2\Bigr)^{1/2},
\]
respectively. 

We now consider two examples of inverse regression problems, 
namely \emph{numerical differentiation} and \emph{deconvolution on $\R$}. 
For these sampling models, we study the frequentist properties 
of the posterior distribution for standard prior that have not been considered 
for inverse regression problems so far.


\subsection{Numerical differentiation using spline prior}
\label{sec:splines}


In this section, we consider the inverse regression problem (\ref{eq:regression:model}) 
with the operator $K$ between $L_1[0,1]$ and the space of functions differentiable 
almost everywhere on the interval $[0,1]$ (see also Chapter 7 of \citep{RudinAnalysis}) defined by
\begin{equation}
	Kf(x) = \int_0^x f(t) dt, \qquad \text{for } x \in [0,1]. 
\label{eq:model:spline:voltera}
\end{equation}
We note that the operator $K$ is not defined between two Hilbert spaces, 
hence goes beyond the concept of singular value decomposition. 
This model is particularly useful for numerical differentiation, for instance,  
and has been well studied in the literature. In particular, in \citep{cavalier2008inverse} 
a related problem of estimating a derivative of a square integrable function 
is presented and it is shown that the SVD basis is the Fourier basis. 
Moreover, the operator is mildly ill-posed of degree $1$ (cf.\ Section \ref{sec:sequence:model}).  
We consider a prior on $f$ that is well-suited if the true regression function 
$f_0$ belongs to the H\"older space $\mathcal{H}(\beta,L)$ for some $\beta >0$, 
that is $f_0$ is $\beta_0 = \lfloor \beta \rfloor$ times differentiable and
\[
	\|f_0\|_\beta = 
	\sup_{x \neq y} \frac{|f^{(\beta_0)}(x) - f^{(\beta_0)}(y)|}{|x-y|^{\beta - \beta_0}} \leq L.  
\]
Since $Kf_0$ is $(\beta_0 +1)$ times differentiable, 
it also holds that $f_0 \in \mathcal{H}(\beta,L)$ implies 
then $Kf_0 \in \mathcal{H}(\beta +1,L)$.

We construct a prior on $f$ by considering its decomposition in  
a B-spline basis. A definition of the B-spline basis can be found in 
\citep{de1978practical}. For a fixed positive integer $q>1$ called 
the degree of the basis, and a given partition of $[0,1]$ in $m$ 
subintervals of the form $((i-1)/m,i/m]$, the space of splines 
is a collection of function $f (0,1] \to \R$ that are $q-2$ times 
differentiable and if restricted to one of the sets $((i-1)/m,i/m]$, 
are polynomial of degree at most $q$. An interesting feature of the space 
of splines is that it forms a $J$-dimensional linear space with 
the so called B-spline basis denoted $(B_{1,q}, \dots, B_{J,q})$, 
for $J = m+q-1$. 
Priors based on the decomposition of the function $f$ in the B-spline 
basis of order $q$ have been considered in the regression setting in, e.g.,  
\citep{GhosalVdv07} and \citep{ShenGhosalAdaptive2014}, and are commonly 
used in practice. Here we construct a different version of the prior that 
will prove to be useful to derive contraction rate for the direct problem 
and the inverse problem. 

Let the prior distribution on $f$ be defined as follows:
\begin{equation}
	\Pi := 
	\begin{cases}
J \sim \Pi_J  \\
a_1, \dots a_J \overset{iid}{\sim} \Pi_{a,J}  \\
f(x) = {J} \sum_{j=1}^{J-1} (a_{j+1} - a_j) B_{j,q-1}(x). 
	\end{cases}
\label{eq:regression:prior}
\end{equation}
Given the definition of $B_{j,q}$ in \citep{de1978practical}, standard computations give
\[
	B_{j,q}'(x) = {J} \bigl( B_{j,q-1}(x) - B_{j+1,q-1}(x) \bigr)
\]
which in turn gives
\begin{equation}
	Kf(x) = \sum_{j=1}^{J} a_j B_{j,q}(x).
\label{eq:reg:voltera:Kf}
\end{equation}
This explains why we choose a prior as in (\ref{eq:regression:prior}) since it 
leads to the usual spline prior on $Kf$. Note that the condition that 
$Kf(0) = 0$ can be imposed by a specific choice of nodes for the 
B-spline basis (see \citep{de1978practical}).
To compute the modulus of continuity for this model, we need to impose 
some conditions on the design. Let $\Sigma_n^q$ be a matrix 
defined by its elements
\[
	(\Sigma_n^q)_{i,j} = \frac{1}{n} \sum_{l=1}^n B_{i,q}(x_l) B_{j,q}(x_l), \quad i,j = 1, \dots, J.
\]
Similarly to \citep{GhosalVdv07}, we ask that the design points satisfy 
the following conditions: 
\begin{description}
\item[\bf D1] for all $\v_1 \in \R^J$
\[
J^{-1} \|\v_1\|^2_{J} \asymp \v_1' \Sigma_n^q \v_1
\]
\item[\bf D2] for all $\v_2 \in \R^{J-1}$ 
\[
(J-1)^{-1} \|\v_2\|^2_{J-1} \asymp \v_2' \Sigma_n^{(q-1)} \v_2.
\]
\end{description}
Condition {\bf D1} is natural when considering B-splines priors 
in a regression setting, and both conditions are satisfied for 
a wide variety of designs. 
Consider for instance the uniform design $x_i = i/n$ for $i = 1, \dots, n$. 
Then given Lemma 4.2 in \citep{ghosal2000convergence}, 
we get that for $\v_1 \in \R^J$, $\v_2 \in \R^{J-1}$ 
\[
	\|\v_1\|_J^2 J^{-1} \lesssim 
		\Bigl\| \sum_{j=1}^J \v_{1,j} B_{j,q} \Bigr\|^2 
		\lesssim \|\v_1\|_J^2 J^{-1},
\]
\[
	\|\v_2\|_{J-1}^2 (J-1)^{-1} \lesssim 
		\Bigl\| \sum_{j=1}^{J-1} \v_{2,j} B_{j,q-1} \Bigr\|^2 
		\lesssim \|\v_2\|_{J-1}^2 (J-1)^{-1},
\]
and the constants depend only on $q$.
Furthermore, we have that 
\[
	\Bigl\| \sum_{j=1}^{J} \v_{1,j} B_{j,q} \Bigr\|^2 = \v_1' \Sigma_n^q \v_1 + O\bigl(n^{-1}\bigr),
\]
where the remainder depends only on $q$. Similarly, 
\[
	\Bigl\| \sum_{j=1}^{J-1} \v_{2,j} B_{j,q-1} \Bigr\|^2 = \v_2' \Sigma_n^{q-1} \v_2 + O\bigl(n^{-1}\bigr). 
\]
Thus \textbf{D1} and \textbf{D2} are satisfied for the uniform design for all $J=o(n)$.

We now go on and derive conditions on the prior such that the posterior 
contracts at the minimax adaptive rate (up to a $\log n$ factor). 
The prior we consider is not conjugate, and does not depend on the singular 
value decomposition of the operator $K$ for obvious reasons. 

\begin{theorem}
Let $Y^n = (Y_1, \dots, Y_n)$ be a sample from (\ref{eq:regression:model}) 
with $\mathcal{X} = [0,1]$ and $\Pi$ be a prior for $f$ as in (\ref{eq:regression:prior}). 
Suppose that $\Pi_J$ is such that for some constants $c_d, c_u > 0$ and $t\geq 0$, 
\begin{equation}
	\exp(-c_d j (\log j)^t) \leq \Pi_J(j\leq J \leq 2 j), \quad  
	\Pi_J(J>j) \lesssim \exp(-c_u j (\log j)^t),
\label{eq:regression:volterra:conditionJ}
\end{equation}
for all $J>1$, and suppose that $\Pi_{a,J}$ is such that for all 
$a_0 \in \R^J$, $\|a_0\|_\infty \leq H$, there exists a constant $c_2$ 
depending only on $H$ such that
\begin{equation}
	\Pi_{a,J}(\|a-a_0\|_J \leq \epsilon) \geq \exp(-c_2 J \log(1/\epsilon))
\label{eq:regression:volterra:conditionPi_J}
\end{equation}
Let $\Theta(\beta,L,H) = \{ f \in \mathcal{H}(\beta,L), \|f\|_\infty \leq H\}$.
If the design $(x_1,\dots,x_n)$ satisfies conditions \textbf{D1} and \textbf{D2}, 
then for all $L$ and for all $\beta \leq q$ 
there exits a constant $C>0$ that depends only on $q$, $L$, $H$ and $\Pi$ 
such that if $f_0 \in \mathcal{H}(\beta,L)$, then 
\begin{equation}
     \sup_{\beta \leq q-1}	\sup_{f_0 \in \Theta(\beta,L,H) } 
     \E_0  \Pi\bigl( \|f - f_0\|_n \geq  C n^{-\frac{\beta}{2 \beta +3}} (\log n)^{3r}  \bigm| Y^n \bigr) \to 0, 
	\label{eq:regression:thvoltera:result}
\end{equation}
with $r = (1\vee t)(\beta+1)/(2\beta + 3)$.
\label{th:regression:voltera}
\end{theorem}

Condition (\ref{eq:regression:volterra:conditionJ}) is for instance 
satisfied by the Poisson or geometric distribution. A similar condition 
is considered in \citep{ShenGhosalAdaptive2014}. 
Condition (\ref{eq:regression:volterra:conditionPi_J}) is satisfied for usual 
choices of priors, such as the product of $J$ independent copies of a distribution 
that admits a continuous density. Similar results hold for functions that 
are not uniformly bounded, with additional conditions on the tails of 
$\Pi_{a,J}$. This will only require additional computations similar 
to those in \citep{ShenGhosalAdaptive2014}, and will thus not be treated here.   

This theorem gives theoretical validation for a family of priors that are widely used in practice for regression problems and are easy to implement. A key feature here is that we can control the transformation of a spline basis function by the operator $K$ through \eqref{eq:reg:voltera:Kf}, which in turn allows us to control the change of norms. This point is highly interesting as it gives guidelines for the construction of priors for inverse problems. Namely, it suggests that a prior whose geometry does not change too much through the application of the operator $K$ could lead to optimal contraction for the inverse problem.  


\subsection{Deconvolution using mixture priors}
\label{sec:mixtures}



In this section, we consider the model (\ref{eq:regression:model}), 
where $K$ is the convolution operator in $\R$. This model is 
widely used in practice, especially when considering auxiliary 
variables in a regression setting or for image deblurring. 
For a convolution kernel $\lambda \in L_2(\R)$ symmetric around 0, 
and for all $f \in L_2(\R)$, we define $K$ as
\begin{equation}
	Kf(x) = \lambda \star f(x)  = \int_\R f(u) \lambda(x-u) du ,~ \qquad \text{for } x \in \R.
\label{eq:regression:convol:operator}
\end{equation}  
To the authors' best knowledge, theoretical properties of Bayesian 
nonparametric approach to this nonparametric regression model have not been studied in the literature. 
In this setting we consider a mixture type prior on $f$, 
and derive an upper bound for the posterior contraction rate. 
Mixture priors are common in the Bayesian literature: \citep{vdv2001}, 
\citep{GhosalVdv07} and \citep{SGmulti12} consider mixtures of Gaussian 
kernels, \citep{kruijer:rousseau:vdv:09} consider location scale mixture 
and \citep{rousseau2010beta} studies mixtures of betas. 
Nonetheless, since they do not fit well into the usual setting based 
on the SVD of the operator, mixture priors have not be considered 
in the literature for ill-posed inverse problems. In our case, 
they proved particularly well suited for the deconvolution problem. 

Let $Y^n = (Y_1, \dots, Y_n)$ be sampled from model (\ref{eq:regression:model}) 
for a true regression function $f_0 \in L_2(\R)$ with $\mathcal{X} = \R$, 
and assume that for $c_x>0$, for all $i =1, \dots, n$, 
$x_i \in [-c_x \log n, c_x \log n]$. It is equivalent to imposing tail 
conditions on the design distribution in the random design setting. 
We choose a prior that is well suited for $f_0$ in the Sobolev ball $W^\beta(L)$, for 
some $\beta > 0$. To avoid technicalities, we will also assume that 
$f_0$ has finite support, that we may choose 
to be $[0,1]$ without loss of generality. Similar results should hold 
for function with support on $\R$ with additional assumptions on the 
tails of $f_0$ but are not treated here.

For a collection of kernels $\Psi_v$ that depend on the parameter $v$, 
a positive integer $J$ and a sequence of nodes $(z_1,\dots,z_J)$ 
we consider the following decomposition of the regression function $f$ 
from the model (\ref{eq:regression:model})
\[
	f(\cdot) = \sum_{j=1}^J w_j \Psi_v(\cdot - z_j),
\]
where $(w_1,\dots,w_J)\in \R^J$ is a sequence of weights. We choose $\Psi_j$ 
proportional to a Gaussian kernel of variance $v^2$ and the uniform sequence 
of nodes $z_j = j/J$ for $j$ such that $j/J \in [-2c_x \log n, 2c_x \log n]$
\[
	\Psi_{j,v}(x) = \Psi_v(x - z_j) = \frac{1}{\sqrt{2\pi v^2}} e^{-\frac{(x-j/J)^2}{2v^2}},
\]
The choice of a Gaussian kernel is fairly natural in the nonparametric literature. 
In our specific case it will prove to be particularly well suited. 
The main advantage of Gaussian kernels in this case is that we can easily compute 
the Fourier transform of $f$ and thus use a similar approach as in 
Section \ref{sec:modulus} to control the modulus of continuity.
We consider the following prior distribution on $f$
\begin{equation}
	\Pi :=
	\begin{cases}
		J \sim \Pi_J \\ 
		v \sim \Pi_v \\
		w_1, \dots,w_J|J \sim \bigotimes_{j=1}^J N(0,1) 
	\end{cases}
\label{eq:mixture:def:prior}
\end{equation}
We use a specific Gaussian prior for the weights $(w_1,\dots,w_J)$ 
in order to use the results on Reproducing Kernel Hilbert Spaces 
following \citep{dejongvzanten12} to derive contraction rate for 
the direct problem. However, we believe that the following result 
should hold for more general classes of priors, but the computations 
would be more involved. 

Following \citep{fan1991optimal}, we define the degree of ill-posedness 
of the problem through the Fourier transform of the convolution kernel. 
For $p>0$, we say that the problem is mildly ill-posed of degree $p$ 
if there exist some constants $c,C>0$ such that for $\hat{\lambda}$, 
the Fourier transform of $\lambda$, 
\[
	\hat{\lambda}(t) = \int \lambda(u) e^{itu} du,
\]
we have for $|t|$ sufficiently large
\begin{equation}
	c |t|^{-p} \leq |\hat{\lambda}(t)| \leq C |t|^{-p}, \quad p \in \N^*,
\label{eq:regressim:deconv:illpos}
\end{equation}
For all $f_0 \in W^\beta(L)$, we have that $Kf_0 \in W^{\beta +p}(L')$ for $L' = LC$. 
Under these conditions, the following Theorem gives an upper bound on 
the posterior contraction rate. 

\begin{theorem}
Let $Y^n = (Y_1, \dots, Y_n)$ be sampled from (\ref{eq:regression:model}) 
with $\mathcal{X} = \R$ and assume that the design satisfies 
$(x_1, \ldots, x_n) \in [-c_x\log n, c_x\log n]^n$. 
Let $f_0$ be such that for $\beta \in \N^*$ 
and $M > 0$, $f_0 \in W^\beta(L)$ with support on $[0,1]$ and 
$\|f_0\|_\infty \leq M$. Consider $K$ as in (\ref{eq:regression:convol:operator}) 
with $\lambda$ satisfying (\ref{eq:regressim:deconv:illpos}). 
Let $\Pi$ be a prior distribution as in (\ref{eq:mixture:def:prior}) 
with 
\[
	\Pi_J(J=j) \asymp j^{-s},
\]
\[
	v^{-q}\exp\Bigl(-\frac{c_d}{v} \log(1/v)^u \Bigr) 
			\lesssim \Pi_v(v) 
			\lesssim v^{-q}\exp\Bigl(-\frac{c_u}{v} \log(1/v)^u\Bigr),
\]
for some positive constants $s$, $c_u$, $c_d$, $q$, and $u$. 
Then there exist constants $C$ and $r$ depending only on $\Pi$, $L$, $K$ and $M$ such that 
\[
	\E_0 \Pi\bigl(\|f-f_0\| \geq Cn^{-\frac{\beta}{1+2\beta + 2p}} (\log n)^r \bigm| Y^n) \to 0.
\] 
\label{th:regression:deconvolution}
\end{theorem}

Note that the prior does not depend on the regularity $\beta$ of the true $f_0$ 
and the posterior contracts at the minimax rate. Our approach is thus adaptive. 
Moreover, the prior does not depend on the degree of ill-posedness either. 
It is thus well suited for a wide variety of convolution kernels. 
In particular, this can be useful when the operator is only partially known, 
as in this case when the regularity of the kernel may not be accessible. 
However, this is beyond the scope of this article.

We prove Theorem \ref{th:regression:deconvolution} by applying Theorem \ref{thm:main} 
together with Lemma \ref{lem:main}. A first difficulty is to define the sets $\mathcal{S}_n$ 
on which we can control the modulus of continuity. A second problem is to derive 
the posterior contraction rate for the direct problem, given that in our setting 
$Kf$ is supported on the real line: 
\citep{dejongvzanten12} derived the posterior contraction rate only for H\"older smooth functions 
with bounded support. However, their results directly extend to the case of convolution of 
Sobolev functions with bounded support given the results of \citep{scricciolo2014adaptive}. 
The complete proof of this Theorem is postponed to Section \ref{sec:proof:reg:deconv}.




\section{Proofs} 
\label{sec:proofs}
\subsection{Proof of the main theorem}
\label{sec:proof:main}
\begin{proof}[Proof of Theorem \ref{th:main}]
By the definition of the modulus of continuity
\begin{align*}
	&\E_0\Pi\bigl(f: d(f,f_0) \geq \omega(\mathcal{S}_n,f_0,d, d_K,M_n\epsilon_n)\, | \,  Y^n\bigr)\\
    &\qquad\qquad \leq \E_0\Pi\bigl(f\in\mathcal{S}_n: d(f,f_0) \geq \omega(\mathcal{S}_n, f_0, d, d_K,M_n\epsilon_n)\, | \,  Y^n\bigr) 
		+ \E_0\Pi(\mathcal{S}_n^c \, | \, Y^n)\\
    &\qquad\qquad  \leq \E_0\Pi\bigl(f\in\mathcal{S}_n: d_K(Kf,Kf_0) \geq M_n \epsilon_n\, | \,  Y^n\bigr) 
        + \E_0\Pi(\mathcal{S}_n^c \, | \, Y^n).
\end{align*}
Together with (\ref{eq:postmass}) and (\ref{eq:mainth:condition:contraction:direct}) it completes the proof.
\end{proof}  

\subsection{Proofs of Section~\ref{sec:sequence:model}} 

\subsubsection{Mildly ill-posed problems}
\label{sec:proof:mildly}

\begin{proof}[Proof of Theorem \ref{th:mildly}]
We first note that if $\|f\|_\beta \leq R$, then $\|Kf\|_{\beta+p} \leq CR$. 
Next we verify the condition of Lemma~\ref{lem:main}. Let 
\[
	k_n = n^{\frac{1}{1+2\alpha+2p}}, \quad 
	\rho_n = n^{-\frac{(\alpha \wedge \beta)}{1+2\alpha+2p}}, \quad
	\epsilon_n = n^{-\frac{(\alpha \wedge \beta)+p}{1+2\alpha+2p}}.
\]
Note that
\[
	n\epsilon_n^2 = n\cdot n^{-\frac{2(\alpha \wedge \beta)+2p}{1+2\alpha+2p}} 
	= n^{\frac{1+2\alpha - 2(\alpha \wedge \beta)}{1+2\alpha+2p}}
	= \epsilon_n^{-\frac{1+2\alpha-2(\alpha\wedge\beta)}{(\alpha\wedge\beta)+p}}, 
\]
hence $\Pi(B_n(Kf_0, \epsilon_n))\gtrsim \exp(-C_2n\epsilon_n^2)$ uniformly over 
a Sobolev ball of radius $R$ (see Lemma~~\ref{lem:mildlyKL} at the end of this subsection). 

Note also that 
\[
	\rho_n^2k_n^{1+2\alpha} = n^{-\frac{2(\alpha \wedge \beta)}{1+2\alpha+2p}}\cdot n^{\frac{1+2\alpha}{1+2\alpha+2p}}
	= n^{\frac{1+2\alpha - 2(\alpha \wedge \beta)}{1+2\alpha+2p}}
	= n\epsilon_n^2, 
\]
and given $c \geq 2(1+2\alpha)/\alpha$ we have $\Pi(\mathcal{S}_n^c) \leq \exp(-(c/8)n\epsilon_n^2)$ 
by Lemma~\ref{lem:mildlypriormass}.

Hence
\[
	\frac{\Pi(\mathcal{S}_n^c)}{\Pi(B_n(Kf_0, \epsilon_n))} \lesssim \exp\Bigl(-\Bigl(\frac{c}{8}-C_2\Bigr)n\epsilon_n^2\Bigr),
\]
uniformly over a ball of radius $R$. The condition of Lemma~\ref{lem:main} is verified upon 
choosing $c = 8(2+C_2)\vee 2(1+2\alpha)/\alpha$.

Finally, we note that (cf.\ (\ref{eq:mod}))
\begin{align*}
	\omega(\mathcal{S}_n,f_0, \|\cdot\|, &\|\cdot\|,M_n\epsilon_n) \\
		&\lesssim M_n n^{\frac{p}{1+2\alpha+2p}}\cdot n^{-\frac{(\alpha \wedge \beta)+p}{1+2\alpha+2p}} +
		n^{-\frac{(\alpha \wedge \beta)}{1+2\alpha+2p}} + n^{-\frac{\beta}{1+2\alpha+2p}}\\
		&\lesssim M_n n^{-\frac{(\alpha \wedge \beta)}{1+2\alpha+2p}},
\end{align*}
which ends the proof.
\end{proof}

\begin{lemm}
\label{lem:mildlyKL}
Suppose $f_0 \in S^\beta$. Then for every $R > 0$ there exist 
positive constants $C_1, C_2$ such that for all $\epsilon \in (0, 1)$, 
\[
    \inf_{\|f_0\|_\beta\leq R}\Pi(B_n(Kf_0,\epsilon))
  \geq C_1\exp\Bigl(-C_2\epsilon^{-\frac{1+2\alpha-2(\alpha\wedge\beta)}{(\alpha\wedge\beta)+p}}\Bigr).
\]
\end{lemm}

\begin{proof}
This proof is adapted from \citep{Belitser}. 
Recall that in the white noise model the $\ell_2$ balls and Kullback--Leibler
neighborhoods are equivalent. By independence, for any $N$, 
\begin{equation}
\label{eq:twoterms}
\begin{split}
    \Pi\Bigl(&\sum_{i=1}^\infty(\kappa_if_i-\kappa_if_{0,i})^2\leq \epsilon^2\Bigr)\\
        &\geq \Pi\Bigl(\sum_{i=1}^N(\kappa_if_i-\kappa_if_{0,i})^2\leq \epsilon^2/2\Bigr)
            \Pi\Bigl(\sum_{i=N+1}^\infty(\kappa_if_i-\kappa_if_{0,i})^2\leq \epsilon^2/2\Bigr).   
\end{split}
\end{equation}
Also
\begin{equation}
\label{eq:tailhelp}
    \sum_{i=N+1}^\infty(\kappa_if_i-\kappa_if_{0,i})^2 \leq 2\sum_{i=N+1}^\infty\kappa_i^2f_i^2
        + 2\sum_{i=N+1}^\infty\kappa_i^2f_{0,i}^2.
\end{equation}
The second sum in the display above is less than or equal to
\[
    2N^{-2\beta-2p}\sum_{i=N+1}^\infty i^{2\beta}f_{0,i}^2 
        \leq 2N^{-2\beta-2p}\|f_0\|_\beta^2 < \frac{\epsilon^2}{4},
\]
whenever $N > N_1 = (8\|f_0\|_\beta^2)^{1/(2\beta+2p)}\epsilon^{-1/(\beta+p)}$.

By Chebyshev's inequality, the first sum on the right-hand side of 
(\ref{eq:tailhelp}) is less than $\epsilon^2/4$ with probability at least
\[
    1-\frac{8}{\epsilon^2}\sum_{i=N+1}^\infty \E_{\Pi}(\kappa_i^2f_i^2) 
     = 1- \frac{8}{\epsilon^2}\sum_{i=N+1}^\infty i^{-1-2\alpha-2p} 
     \geq 1 - \frac{4}{(\alpha+p)N^{2(\alpha+p)}\epsilon^2}> 1/2
\]
if $N > N_2 = (8/(\alpha+p))^{1/(2\alpha+2p)}\epsilon^{-1/(\alpha+p)}$.

To bound the first term in (\ref{eq:twoterms}) we apply Lemma~6.2 in \cite{Belitser} 
with $\xi_i = \kappa_if_{0,i}$ and $\delta^2=\epsilon^2/2$. Note that
\begin{align*}
	\sum_{i=1}^Ni^{1+2\alpha+2p}\xi_i^2 &= \sum_{i=1}^Ni^{1+2\alpha+2p}\cdot i^{-2p}f_{0,i}^2 \\
	  &= \sum_{i=1}^Ni^{1+2\alpha-2\beta}f_{0,i}^2i^{2\beta} \leq N^{(1+2\alpha-2\beta)\vee 0}\|f_0\|^2_\beta.
\end{align*}
Therefore,
\begin{align*}
    \Pi\Bigl(\sum_{i=1}^N(\kappa_if_i-&\kappa_if_{0,i})^2 \leq \epsilon^2/2\Bigr)\\
      &\geq \exp\Bigl(-\Bigl(1+2\alpha+2p+\frac{\log 2}{2}\Bigr)N\Bigr)
          \exp\Bigl(-N^{(1+2\alpha-2\beta)\vee 0}\|f_0\|_\beta^2\Bigr)\\
      &\qquad \times{}
          \Pr\Bigl(\sum_{i=1}^N V_i^2 \leq 2\delta^2N^{1+2\alpha+2p}\Bigr).
\end{align*}
The last term, by the central limit theorem, is at least $1/4$ if $2\delta^2N^{1+2\alpha+2p}>N$ 
and $N$ is large, that is, $N > N_3 = \epsilon^{-1/(\alpha+p)}$ and $N > N_4$, where $N_4$ 
does not depend on $f_0$. Choosing $N = \max\{N_1, N_2, N_3, N_4\}$ we obtain
\begin{align*}
    \Pi(f:\|Kf-&Kf_0\|\leq \epsilon)\\
      &\geq \frac{1}{8}\exp\Bigl(-\Bigl(1+2\alpha+2p+\frac{\log 2}{2}\Bigr)N\Bigr)
          \exp\Bigl(-N^{(1+2\alpha-2\beta)\vee 0}\|f_0\|_\beta^2\Bigr).
\end{align*}
Consider $\alpha \geq \beta$. Then $\exp(-N) \geq \exp(- N^{(1+2\alpha-2\beta)})$ so 
\[
    \Pi(f:\|Kf-Kf_0\|\leq \epsilon) \geq \frac{1}{8}
          \exp\Bigl(-C_3N^{(1+2\alpha-2\beta)}\Bigr),
\]
for some constant $C_3$ that depends only on $\alpha, \beta, p$ and $\|f_0\|_\beta^2$. Moreover, 
since $\epsilon < 1$ and $\alpha \geq \beta$, $N$ is dominated by $\epsilon^{-1/(\beta+p)}$ and 
we can write
\[
    \Pi(f:\|Kf-Kf_0\|\leq \epsilon) \geq \frac{1}{8}
          \exp\Bigl(-C_4\epsilon^{-\frac{1+2\alpha-2\beta}{\beta+p}}\Bigr), 
\]
where $C_4$ depends on $f_0$ again through $\|f_0\|_\beta^2$ only. 

Now consider $\alpha < \beta$. Similar arguments lead to 
\[
    \Pi(f:\|Kf-Kf_0\|\leq \epsilon) \geq \frac{1}{8}
          \exp\Bigl(-C_5\epsilon^{-\frac{1}{\alpha+p}}\Bigr), 
\]
for some constant $C_5$ that depends only on $\alpha, \beta, p$ and $\|f_0\|_\beta^2$.
\end{proof}

\begin{lemm}
\label{lem:mildlypriormass}
Let $\rho_n$ be an arbitrary sequence tending to $0$,  $c$ be an arbitrary constant, 
and let the sequence $k_n \to \infty$ satisfy $k_n^{2\alpha} \geq 2(1+2\alpha)/(\alpha c\rho_n^2)$.
Then
\[
    \Pi(\mathcal{S}_n^c) \leq \exp\Bigl(-\frac{c}{8}\rho_n^2k_n^{1+2\alpha}\Bigr).
\]
\end{lemm}

\begin{proof}
For $W_1, W_2, \ldots$ independent standard normal random variables
\[
    \Pi(\mathcal{S}_n^c) = \Pr\Bigl(\sum_{i>k_n} \lambda_iW_i^2
            > c\rho_n^2\Bigr).
\]
For some $t > 0$
\begin{align*}
    \Pr\Bigl(&\sum_{i>k_n} \lambda_iW_i^2
            > c\rho_n^2\Bigr)\\
        &= \Pr\Bigl(\exp\Bigl(t\sum_{i>k_n} \lambda_iW_i^2\Bigr)
            > \exp(tc\rho_n^2)\Bigr)
        \leq \exp(-tc\rho_n^2)
            \E \exp\Bigl(t\sum_{i>k_n} \lambda_iW_i^2\Bigr)\\
        &= \exp(-tc\rho_n^2)
            \prod_{i>k_n} \E\exp(t\lambda_iW_i^2)
        = \exp(-tc\rho_n^2)
            \prod_{i>k_n} (1-2t\lambda_i)^{-1/2}.
\end{align*}
We first applied Markov's inequality, and later used properties of
the moment generating function. Here we additionally assume
that $2t\lambda_i < 1$ for $i > k_n$.

We take the logarithm of the right-hand side of the
previous display. Since $\log(1-y)\geq -y/(1-y)$, we have
\begin{align*}
    -tc\rho_n^2 + &\sum_{i>k_n}\log (1-2t\lambda_i)^{-1/2}\\
        &=  -tc\rho_n^2
            - \frac{1}{2} \sum_{i>k_n}\log (1-2t\lambda_i)
        \leq -tc\rho_n^2
            + \frac{1}{2} \sum_{i>k_n}\frac{2t\lambda_i}{1-2t\lambda_i}.
\end{align*}
We continue with the latter term, noticing that $1-2t\lambda_i > 1-2tk_n^{-1-2\alpha}$ 
for $i > k_n$
\[
	\frac{1}{2} \sum_{i>k_n}\frac{2t\lambda_i}{1-2t\lambda_i} \leq 
	   \frac{t}{1-2tk_n^{-1-2\alpha}} \sum_{i>k_n}i^{-1-2\alpha}.
\]
Since $x^{-1-2\alpha}$ is decreasing, 
we have that
\[
	\sum_{i>k_n}i^{-1-2\alpha} \leq \int_{k_n}^\infty x^{-1-2\alpha}\, dx + k_n^{-1-2\alpha} 
	  = \frac{k_n^{-2\alpha}}{2\alpha} + k_n^{-1-2\alpha} \leq k_n^{-2\alpha} \frac{1+2\alpha}{2\alpha}, 
\]
noting that $k_n > 1$ for $n$ large enough. Finally
\[
	    -tc\rho_n^2
            + \sum_{i>k_n}\log (1-2t\lambda_i)^{-1/2} 
            \leq -tc\rho_n^2
            + \frac{1+2\alpha}{2\alpha}
                \frac{t}{1-2tk_n^{-1-2\alpha}} k_n^{-2\alpha}.
\]
Thus for $t = k_n^{1+2\alpha}/4$
\[
    \Pi(\mathcal{S}_n^c)\leq \exp\Bigl(-\frac{c}{4}\rho_n^2k_n^{1+2\alpha} + \frac{1+2\alpha}{4\alpha}k_n\Bigr) 
	\leq \exp\Bigl(-\frac{c}{8}\rho_n^2k_n^{1+2\alpha}\Bigr), 
\]
since $k_n^{2\alpha} \geq 2(1+2\alpha)/(\alpha c\rho_n^2)$.
\end{proof}

\subsubsection{Severely and extremely ill-posed problems}
\label{sec:proof:severly}
\begin{proof}[Proof of Theorem \ref{th:WN:extremely}]
Assume for brevity that we have the exact equality $\kappa_i = \exp(-\gamma i^p)$.
Dealing with the general case is straightforward, but makes the proofs somewhat
lengthier.

Since $Y_i | f_i \sim N(\kappa_if_i, n^{-1})$ and $f_i \sim N(0, \lambda_i)$
for $i \leq k_n$, the posterior distribution (for $Kf$) can be written as 
$(Kf)_i |Y^n \sim N(\sqrt{nt_{i,n}}Y_i,v_{i,n})$ for $i \leq k_n$, where
%
%
%
%
%
%
%
\[
    v_{i,n} = \frac{\lambda_i\kappa_i^2}{1+n\lambda_i\kappa_i^2},
    \qquad
    t_{i,n} = \frac{n\lambda_i^2\kappa_i^4}{(1+n\lambda_i\kappa_i^2)^2}.
\]
Since the posterior is Gaussian, we have
\begin{equation}\label{eq:postrisk}
    \int\|Kf-Kf_0\|^2\, d\Pi(Kf|Y^n) = \|\widehat{Kf}-Kf_0\|^2+\sum_{i\leq k_n} v_{i,n},
\end{equation}
where $\widehat{Kf}$ denotes the posterior mean and can be rewritten as:
\[
\begin{split}
    \widehat{Kf} &= \Bigl(\frac{n\lambda_i\kappa_i^2}{1+n\lambda_i\kappa_i^2}Y_i\Bigr)_{i=1}^{k_n}
    = \Bigl(\frac{n\lambda_i\kappa_i^3f_{0,i}}{1+n\lambda_i\kappa_i^2}
      + \frac{\sqrt{n}\lambda_i\kappa_i^2Z_i}{1+n\lambda_i\kappa_i^2}\Bigr)_{i=1}^{k_n}\\
    &=: \E\widehat{Kf} + (\sqrt{t_{i,n}}Z_i)_{i=1}^{k_n}.
\end{split}
\]

By Markov's inequality the left side of (\ref{eq:postrisk}) is an upper bound to
$M_n^2\epsi_n^2$ times the desired posterior probability. Therefore, in order to show that
$\Pi(f: \|Kf-Kf_0\| \geq M_n\epsi_n|Y^n)$ goes to zero in probability, it suffices to show
that the expectation (under the true $f_0$) of the right hand side of (\ref{eq:postrisk}) is
bounded by a multiple of $\epsi_n^2$. The last term is deterministic. As for the first
term we have
\[
    \E\|\widehat{Kf}-Kf_0\|^2 = \|\E\widehat{Kf}-Kf_0\|^2 + \sum_{i\leq k_n} t_{i,n}.
\]
We also observe
\[
    \|\E\widehat{Kf}-Kf_0 \|^2 
	   = \sum_{i\leq k_n} \frac{\kappa_i^2f_{0,i}^2}{(1+n\lambda_i\kappa_i^2)^2} 
	      + \sum_{i > k_n} \kappa_i^2 f_0^2.
\]
Note that $t_{i,n} \leq n^{-1}$ and $s_{i,n} \leq n^{-1}$, hence
\[
		\sum_{i\leq k_n} v_{i,n} \lesssim n^{-1}k_n \asymp n^{-1}(\log n)^{\frac{1}{p}}, \quad
		\sum_{i\leq k_n} t_{i,n} \lesssim n^{-1}k_n \asymp n^{-1}(\log n)^{\frac{1}{p}}.
\]
By Lemma~\ref{lem:inequalities}
\[
    \sum_{i\leq k_n} \frac{\kappa_i^2f_{0,i}^2}{(1+n\lambda_i\kappa_i^2)^2}
        + \sum_{i > k_n} \kappa_i^2 f_{0,i}^2
    \lesssim \|f_0\|^2_\beta n^{-\frac{2\gamma}{\xi+2\gamma}}(\log n)^{-\frac{2\beta}{p}+\frac{2\gamma\alpha}{p(\xi+2\gamma)}}.
\]
Therefore, the posterior contraction rate for the direct problem is given by
\[
    (\log n)^{-\frac{\beta}{p}+\frac{\gamma\alpha}{p(\xi+2\gamma)}} n^{-\frac{\gamma}{\xi+2\gamma}}, 
\]
and is uniform over Sobolev balls of fixed radius. [This bound is also valid if $\xi = 0$ and $\alpha \geq 1+2\beta$.] 

By (\ref{eq:mod}) an upper bound for the modulus of continuity is given by 
\begin{align*}
	\omega(\mathcal{S}_n,f_0, \|\cdot\|, \|\cdot\|,M_n\epsilon_n) 
	&\lesssim M_n \exp(\gamma k_n^p) \epsilon_n + k_n^{-\beta}\\
	&\lesssim M_n n^{\frac{\gamma}{\xi+2\gamma}} (\log n)^{-\frac{\gamma\alpha}{p(\xi+2\gamma)}} \epsilon_n + (\log n)^{-\frac{\beta}{p}}\\
	&\lesssim M_n (\log n)^{-\frac{\beta}{p}},
\end{align*}
which ends the proof.
\end{proof}

\begin{lemm}
\label{lem:inequalities}
It holds that
\[
    \sum_{i\leq k_n} \frac{\kappa_i^2f_{0,i}^2}{(1+n\lambda_i\kappa_i^2)^2}
        + \sum_{i > k_n} \kappa_i^2 f_{0,i}^2
    \lesssim \|f_0\|^2_\beta n^{-\frac{2\gamma}{\xi+2\gamma}}(\log n)^{-\frac{2\beta}{p}+\frac{2\gamma\alpha}{p(\xi+2\gamma)}}.
\]
\end{lemm}

\begin{proof}
As for the first sum we have
\[
\begin{split}
	\sum_{i\leq k_n} \frac{\kappa_i^2f_{0,i}^2}{(1+n\lambda_i\kappa_i^2)^2} &\leq 
	n^{-2} \sum_{i\leq k_n} \lambda_i^{-2}\kappa_i^{-2}i^{-2\beta}i^{2\beta}f_{0,i}^2\\
	&= n^{-2} \sum_{i\leq k_n} i^{2(\alpha-\beta)}\exp(2(\xi+\gamma)i^p)i^{2\beta}f_{0,i}^2,
\end{split}
\]
and for $k_n$ large enough all terms $i^{2(\alpha-\beta)}\exp(2(\xi+\gamma)i^p)$ are dominated 
by $k_n^{2(\alpha-\beta)}\exp(2(\xi+\gamma)k_n^p)$, so
\begin{equation}
\label{eq:firstsum}
	\sum_{i\leq k_n} \frac{\kappa_i^2f_{0,i}^2}{(1+n\lambda_i\kappa_i^2)^2} \leq
	 n^{-2}k_n^{2(\alpha-\beta)}\exp(2(\xi+\gamma)k_n^p)\|f_0\|^2_\beta.
\end{equation}

As for the second sum we note that
\[
	\sum_{i > k_n} \kappa_i^2 f_{0,i}^2 = \sum_{i > k_n} \exp(-2\gamma i^p)i^{-2\beta}i^{2\beta} f_{0,i}^2,
\]
and since $\exp(-2\gamma i^p)i^{-2\beta}$ is monotone decreasing
\begin{equation}
\label{eq:secondsum}
	\sum_{i > k_n} \kappa_i^2 f_{0,i}^2 
	  \leq \exp(-2\gamma k_n^p)k_n^{-2\beta}\|f_0\|^2_\beta.
\end{equation}

Recall that $\exp(k_n^p) = (nk_n^{-\alpha})^{1/(\xi+2\gamma)}$  
and therefore we can rewrite the bounds in (\ref{eq:firstsum}) and (\ref{eq:secondsum}) 
as
\[
	n^{-2}k_n^{2(\alpha-\beta)}\bigl(nk_n^{-\alpha}\bigr)^{\frac{2(\xi+\gamma)}{\xi+2\gamma}} =  
	n^{-\frac{2\gamma}{\xi+2\gamma}} k_n^{-2\beta + \frac{2\gamma\alpha}{\xi+2\gamma}},
\]
and
\[
	k_n^{-2\beta}\bigl(nk_n^{-\alpha}\bigr)^{-\frac{2\gamma}{\xi+2\gamma}} = 
	n^{-\frac{2\gamma}{\xi+2\gamma}} k_n^{-2\beta + \frac{2\gamma\alpha}{\xi+2\gamma}}.
\]
Finally, since $k_n$ in this case can be taken of the order $(\log n)^{1/p}$, we obtain the 
desired upper bound.
\end{proof}

\subsection{Proofs of Section~\ref{sec:regression}} 

\subsubsection{Numerical differentiation using spline prior} 
\label{sec:proof:voltera}
We first compute an upper bound for the modulus of continuity. 
For $a \in \R^J$ we define $\Delta(a)\in \R^{J-1}$ such that 
$\Delta(a)_i = a_{i+1} - a_i$, for $i = 1, \ldots, (J-1)$. 
Given conditions \textbf{D1} and \textbf{D2} we get,
\[
\begin{split}
	\|f\|_n^2 &= {J^2} \Delta(a)' \Sigma_{n}^{q-1} \Delta(a) 
		\lesssim {J^2} \frac{1}{J-1} \|\Delta(a)\|_{J-1}^2 \\ 
	  &\lesssim {J^2} \frac{1}{J-1} \|a\|_J^2 \lesssim {J^2} \|Kf\|_n^2.
\end{split}
\]
To apply Theorem \ref{th:main}, we first need to derive a contraction rate for $Kf$. 
Note that in this case we simply have a standard non parametric regression 
model with a spline prior. This model has been extensively studied in the 
literature (see, e.g., \citep{GhosalVdv07} or \citep{JongeZantenSplines}) 
and we can easily adapt their results to derive minimax adaptive contraction rates. 

\begin{lemm}
Let $\Pi$ be as in Theorem \ref{th:regression:voltera}. 
Let $Y_n$ be sampled form model \ref{eq:regression:model} 
with $f = f_0$ and assume that $f_0 \in \Theta(\beta,L,H)$ 
with $\beta \leq q-1$. Then there exists a constant $C$ depending only on 
$H$, $L$, $\Pi$, and $q$ such that
\[
	\E_0\Pi\bigl(\|Kf-Kf_0\|_n \geq C n^{-\frac{\beta+1}{2\beta +3}}(\log n)^r \bigm| Y_n\bigr) \to 0
\]
with $r = (1 \vee t)\beta/(2\beta +1)$.  
\label{lem:adapt:cvrate:splines}
\end{lemm}
Similar results have been proved in \citep{ShenGhosalAdaptive2014}, however the 
authors do not give a direct proof of their result. Here this lemma gives us directly 
the posterior contraction rate for the direct problem. 
The proof of this lemma is postponed to the end of this subsection.

\begin{proof}[Proof of Theorem~\ref{th:regression:voltera}]
We now derive the posterior contraction rate of the posterior distribution for 
the inverse problem. We first get an upper bound for the modulus of continuity, 
for $f \in \mathcal{S}_n$. Using standard approximation results on splines 
(e.g. \citep{de1978practical}), we have that for all $J$ there exists $a^0\in \R^J$ 
such that 
\[
	\Bigl\| f_0 -  \sum_{j=1}^{J-1} (a^0_{j+1} - a^0_j) (B_{j,q-1})\Bigr\|_\infty \leq (J-1)^{-\beta} \|f_0\|_\infty,
\]
and 
\[
	\Bigl\|Kf_0 - \sum_{j=1}^J a_j^0 B_{j,q} \Bigr\|_\infty \leq J^{-\beta -1} \|Kf_0\|_\infty.
\]
We thus deduce that for $J \geq 2$, 
\[
\begin{split}
	\|f - f_0\|_n &\leq \|f - f_{a^0}\|_n + \|f_{a^0} - f_0\|_n  \\ 
	&\leq C J^{-1} \| Kf - Kf_n\|_n + \|f_{a^0} - f_0\|_n \\ 
	& \leq C J^{-1} \| Kf - Kf_0\|_n + \|Kf_{a^0} - Kf_0\|_n + \|f_{a^0} - f_0\|_n.
\end{split}
\]

We can thus deduce an upper bound for the modulus of continuity 
\[
	\omega(S_n,f_0,\|\cdot\|_n, \|\cdot\|_n,\delta) \leq J_n \delta.
\]
Applying Theorem \ref{thm:main} gives
\[
	\E_0 \Pi\bigl(\|f - f_0\|_n \geq C n^{-\frac{\beta}{2\beta + 3}} (\log n)^r \bigm| Y^n \bigr) \to 0,  
\]
for a constant $C>0$ depending only on $\|f_0\|_\infty$, $r$, and $\Pi$. 
\end{proof}

\begin{proof}[Proof of Lemma \ref{lem:adapt:cvrate:splines}]
We prove the lemma using Theorem 4 in \citep{GhosalVdv07}. 
Let $\beta \leq q$ and $f_0$ be in $\mathcal{H}(\beta,L)$ and set 
$\epsilon_n = C n^{-(\beta+1)/(2\beta +3)} (\log n)^r$ with $r = (1\vee t)\beta/(2\beta +1)$. 
Set $J_n := J_0 n \epsilon_n^{2} \log(n)^{-t}$ for a fixed constant $J_0>0$ and 
consider the sets $\mathcal{S}_n$ defined by 
\[
	\mathcal{S}_n := \bigl\{ J \leq J_n, a \in \R^J\bigr\}
\]
We first control the local entropy function 
$N(\epsilon, \{ J,a \in \mathcal{S}_n : \|Kf - Kf_0\|_n \leq \epsilon_n \} ,\|\cdot\|_n)$. 
By using the same reasoning as in the proof of Theorem 12 in \citep{GhosalVdv07}, for all $J \in S_n$ we get
\[
	\log(N(\epsilon, \{ J,a \in \mathcal{S}_n : \|Kf - Kf_0\|_n \leq \epsilon_n \} , \|\cdot\|_n) )\leq n \epsilon_n^2.
\]
The prior mass of the set $\mathcal{S}_n$ is easily controlled using condition 
(\ref{eq:regression:volterra:conditionJ}):
\[
	\Pi(\mathcal{S}_n^c) = \Pi_J(J > J_n) \leq \exp(-c_u J_n (\log J_n)^t). 
\]
We now need to control the prior mass of Kullback--Leibler neighborhoods of $Kf_0$. 
Note that this condition will also be useful to apply Lemma \ref{lem:main} and 
thus derive the posterior contraction rate for the direct problem.
Let $B_n(Kf_0, \epsilon)$ be defined as in (\ref{eq:main:de:KLneighbour}). 

Using the results of Section 7.3 in \citep{GhosalVdv07}, setting $\tilde{J}_n = J_n (\log n)^{-r/\beta}$ 
we deduce that for some constant $c$ depending only on $\sigma$ 
\[
	B_n(Kf_0,\epsilon_n) \supset \{ \tilde{J}_n \leq J \leq 2 \tilde{J}_n, \|Kf - Kf_0\|_n^2 \leq c\epsilon_n^2 \}.
\]
Standard approximation results on splines give that for all $J$ there exists a 
sequence $a_0= (a_{0,1}, \dots , a_{0,J})$ such that 
\[
	\Bigl\|Kf_0 - \sum_{j=1}^J a_{0,j} B_{j,q} \Bigr\|_n \leq J^{-\beta -1} \|Kf_0\|_\beta \leq J^{-\beta -1} L.
\]
Given condition {\bf D1} on the design, we thus have that for a constant $c'>0$ 
depending only  on $\sigma$ and $L$
\[
	B_n(Kf_0,\epsilon_n) \supset 
		\bigl\{ \tilde{J}_n \leq J \leq 2 \tilde{J}_n, \|a-a_0\|_{\tilde{J}_n} \leq c' \tilde{J}_n^{1/2} \epsilon_n \bigr\}.
\]
Therefore, we obtain a lower bound on the prior mass of a Kullback--Leibler 
neighbourhood of $Kf_0$:
\[
\begin{split}
	\Pi(B_n(Kf_0,\epsilon_n)) &\geq 
			\Pi\bigl( \tilde{J}_n \leq J \leq 2\tilde{J}_n, \|a - a_0 \|_n \leq  c'\tilde{J}_n^{1/2} \epsilon_n \bigr) \\
		&\geq \exp\bigl(- \tilde{J}_n ( c_d(\log \tilde{J}_n)^t + c_2 \log(\tilde{J}_n^{-1/2}\epsilon_n^{-1}))\bigr).
\end{split}
\]
We thus have for $C_2>0$,
\begin{equation}
	\frac{\Pi(\mathcal{S}_n^c)}{\Pi(B_n(Kf_0,\epsilon_n))} \leq \exp(-C_2 J_n (\log J_n)^t), 
\end{equation}
which together with Theorem 4 in \citep{GhosalVdv07} ends the proof. 
\end{proof}

\subsubsection{Deconvolution using mixture priors} 
\label{sec:proof:reg:deconv}
\begin{proof}[Proof of Theorem \ref{th:regression:deconvolution}]
We first specify the sets $\mathcal{S}_n$ for which we can control the modulus of continuity. 
Denoting $\fha$ the Fourier transform of $f$, for any sequence $a_n$ going to infinity and 
$I_n = [-a_n, a_n]$ we define for $a>0$
\begin{equation}
	\mathcal{S}_n = \Bigl\{f: \int_{I_n} |\fha(t)|^2dt \geq a \int_{I_n^c} |\fha(t)|^2dt  \Bigr\}.
\label{eq:regression:mixtur:def:Sn}
\end{equation}
We control the modulus of continuity $\omega(\mathcal{S}_n,f_0, \|\cdot\|, \|\cdot\|,\delta)$ 
in a similar way as in Section \ref{sec:modulus}. 
First consider $f \in \mathcal{S}_n$, and denote $\fha_n(\cdot) = \fha(\cdot) \I_{I_n}(\cdot)$. 
We then have 
\[
	\|f\|^2 = \|\fha\|^2 \leq (1+a)\|\fha_n\|^2 
	   \lesssim a_n^{2p} \int_{I_n} |\fha|^2 | \hat{\lambda}|^2  
	   \lesssim a_n^{2p} \|Kf\|^2.
\]
Note that for $f_0 \in W^{\beta}(L)$ we have for 
$f_{0,n}(x) = \int \hat{f}_{0,n}(t) e^{-itx} dt$
\[
	\|f_0 - f_{0,n}\| \leq 2 a_n^{-\beta} L, \qquad  \|Kf_0 - Kf_{0,n}\| \leq 2 a_n^{-(\beta +p)}L',
\]
which gives 
\begin{equation}
	\omega(\mathcal{S}_n,f_0, \|\cdot\|, \|\cdot\|,\delta) \lesssim a_n^{p} \delta + a_n^{-\beta}. 
\label{eq:regressim:mixture:modulus}
\end{equation}

We now control the prior mass of $\mathcal{S}_n^c$ in order to apply Lemma~\ref{lem:main}. 
Denote by $l_n = \lfloor a_n /(2\Pi J) \rfloor$, $L_n = \lceil a_n /(2\Pi J) \rceil$. 
We have 
\[
\begin{split}
	\int_{I_n} |\fha(t)|^2 dt &\geq  
			2\pi J \int_{-L_n}^{l_n} e^{-4\pi^2t^2v^2} \Bigl| \sum_{j=1}^J w_j e^{2\pi j t}  \Bigr| dt \\
		& =  2\pi J \sum_{l=-L_n}^{l_n} \int_{l}^{l+1} e^{-4\pi^2t^2v^2} \Bigl| \sum_{j=1}^J w_j e^{2\pi j t}  \Bigr| dt \\ 
		& =  2\pi J \int_{0}^{1}  \Bigl| \sum_{j=1}^J w_j e^{2\pi j t}  \Bigr| \sum_{l=-L_n}^{l_n}  e^{-4\pi^2(t+l)^2v^2} dt \\ 
		& \geq  2\pi J \sum_{l=-L_n}^{l_n}  e^{-4\pi^2(1+|l|)^2v^2}   \int_{0}^{1}  \Bigl| \sum_{j=1}^J w_j e^{2\pi j t}  \Bigr| dt, 
\end{split}
\] 
and similarly we get 
\[
\begin{split}
	\int_{I_n^c} |\fha(t) |^2 dt 
		&\leq 2\pi J   \int_{0}^{1}  \Bigl| \sum_{j=1}^J w_j e^{2\pi j t}  \Bigr| \Bigl(
				\sum_{l=-\infty}^{-L_n}  e^{-4\pi^2(t+l)^2v^2}  
					+ \sum_{l=l_n}^{\infty} e^{-4\pi^2(t+l)^2v^2}\Bigr) dt    \\ 
		& \leq 2\pi J    \Bigl(  \sum_{l=-\infty}^{-L_n}  e^{-4\pi^2l^2v^2}  
			+ \sum_{l=l_n}^{\infty} e^{-4\pi^2l^2v^2}    \Bigr) 
				\int_{0}^{1}  \Bigl| \sum_{j=1}^J w_j e^{2\pi j t}  \Bigr| dt.
\end{split}
\] 
We thus deduce that for absolute constants $C'>0$
\[
	\Pi(\mathcal{S}_n^c) \leq \Pi(v \leq J/a_n ) \lesssim e^{-C'a_n \log a_n}. 
\]
We end the proof by combining this result (choosing $a_n = n\epsilon_n^2$) with Lemma~\ref{lem:main}, 
Lemma~\ref{lem:theotherlemma}, and Theorem~\ref{thm:main}.
\end{proof}

\begin{lemm}\label{lem:theotherlemma}
Let $Y^n$ be sample from \eqref{eq:regression:model} with $K$ defined by \eqref{eq:regression:convol:operator}. Let $\Pi$ be as in Theorem \ref{th:regression:deconvolution}. For all $\beta \in \mathbb{N}^*$ if $f_0\in W^\beta(L)$ with support on $[0,1]$ and $||f_0||_\infty \leq M$, we have for $C>0$ large enough if $\epsilon_n =  n^{-(\beta + p)/(1+2\beta + 2p)} (\log n)^{r}$, where $r$ is some constant, 

\[
	\E_0\Pi(\|Kf - Kf_0\| \geq C  \epsilon_n |Y^n) \to 0,  
\]
and
\[
	\Pi(\|Kf-Kf_0\| \leq  \epsilon_n ) \geq e^{-n\epsilon_n^2}.
\]
\end{lemm}

\begin{proof}
This proof is based on the results of \citep{dejongvzanten12} and \citep{scricciolo2014adaptive}.
We adapt the results of \citep{dejongvzanten12} to our setting in order 
to control the posterior mass of the Kullback--Leibler neighbourhoods of $Kf_0$ 
and the posterior contraction rate for the direct problem. 
Following their notation we have that $K\Psi_v \in \mathcal{P}_\infty$, 
and thus the small ball probability $\Pi(\|f\|_\infty \leq \epsilon)$ 
can be controlled by their Lemma 3.3. We then extend their Lemma 3.5 to our setting. 
Note that with Lemma 9 of \citep{scricciolo2014adaptive}, 
Lemma 3.4 of \citep{dejongvzanten12} holds for the same $T_{\alpha,v}$ with $\alpha = \beta + p$. 
Choosing $h$ to be as in the proof of Lemma 3.5 of \citep{dejongvzanten12} and denoting 
$\omega_0 = f_0 \star \lambda$, we have 
\[
	h(x) = \sum_{j/J \in [-2c_x \log n, 2c_x \log n]} 
			T_{\alpha,v}(\omega_0) \frac{1}{J v} \Psi\Bigl( \frac{x-j/J}{v} \Bigr),
\]
and thus deduce 
\[
	\|h\|^2_{H^{J,v}} \leq 2c_x \|T_{\alpha,v}(\omega_0)\|^2  \log n.
\]
Using their decomposition (3.8), we control 
$|h(x) - \Psi_v \star T_{\alpha,v}(\omega_0)(x)|$ along the same lines 
as in their computations on page 3312. We have 
\[
\begin{split}
	&|h(x) - \Psi_v \star T_{\alpha,v}(\omega_0)(x)|\\
		&\qquad \leq \Bigl|h(x) - \int_{-2c_x \log n}^{2c_x \log n} T_{\alpha,v}(\omega_0)(y) \Psi_v(x-y) dy \Bigr|  \\ 
		&\qquad + \Bigl| \int_{-\infty}^{-2c_x \log n} T_{\alpha,v}(\omega_0)(y) \Psi_v(x-y) dy \Bigr|  
		   + \Bigl|\int_{2c_x \log n}^\infty T_{\alpha,v}(\omega_0)(y) \Psi_v(x-y) dy\Bigr|
%
\end{split}
\]
The first term on the right hand side of the above display can be controlled as in the proof of Lemma 3.5 of \citep{dejongvzanten12}. 
For the last two terms, we have 
\[
\begin{split}
	\Bigl| \int_{-\infty}^{-2c_x \log n} &T_{\alpha,v}(\omega_0)(y) \Psi_v(x-y) dy \Bigr| 
		+  \Bigl|\int_{2c_x \log(n)}^\infty T_{\alpha,v}(\omega_0)(y) \Psi_v(x-y) dy\Bigr| \\ 
		&\lesssim \|T_{\alpha,v}(\omega_0)\|_\infty e^{-\frac{c_x^2 (\log n)^2}{2v^2} } v^{-1}.
\end{split}
\]
Following the same proof of Theorem 2.2 of \citep{dejongvzanten12}, we get
\[
	\E_0\Pi(\|Kf - Kf_0\| \geq C n^{-\frac{\beta + p}{1+2\beta + 2p}} (\log n)^{r} |Y^n) \to 0,  
\]
and similarly to their equation (2.5) we get, 
with $\epsilon_n =  n^{-(\beta + p)/(1+2\beta + 2p)} (\log n)^{r}$, 
where $r$ is some constant, 
\[
	\Pi(\|Kf-Kf_0\| \leq  \epsilon_n ) \geq e^{-n\epsilon_n^2}.
\]

\end{proof}

\section{Discussion}
\label{sec:discussion}

In this paper we propose a new approach to the problem of deriving posterior contraction rates for
linear ill-posed inverse problems. More precisely, we put a prior on the parameter of interest 
$f$ that naturally imposes the prior on $Kf$, leading to a certain rate of contraction in the direct problem. 
Next, we consider a sequence of sets on which the operator $K$ possesses a continuous inverse. Then, 
we impose additional conditions on the prior (or the posterior itself) under which 
the posterior contracts at a certain rate in the inverse problem setting. 

This is a great advantage of the Bayesian approach in this setting as when the posterior
distribution is known to contract at a given rate in the direct problem, one only has to consider
subset of high prior mass for which the norm of the inverse of the operator may be handled. Our
result seems to show that the main difficulty when considering linear inverse problems is to control
the change of metrics form $d_K$ to $d$, which is dealt here by considering the
modulus of continuity as introduced in \citep{donoho1991} and \citep{hoffmann2013adaptive}. 
It is also to be noted that contrariwise to existing methods, we do not require a Hilbertian 
structure for the parameter space, see for instance the example treated in Section \ref{sec:splines}. 
This could be particularly useful when considering nonlinear operators,
and is of potential interest when considering the case of partially 
known operators.  

We recovered (a subset of) the existing results from \citep{KVZ11}, \citep{KVZ13}, \citep{Agapiou12}, 
\citep{Agapiou14}, and \citep{ray2013bayesian}. Our approach should be viewed as a generalization 
of the ideas presented in the last paper 
{and the existing sieve method used in the literature 
on posterior contraction}. Furthermore, we were able to go beyond the sequence setting 
as well as derive posterior contraction rates for prior distributions that were not covered by the existing theory. 
We also treated an operator that does not admit singular value decomposition. In this
sense, the approach proposed in this paper is more general, and we believe more natural, than the 
existing ones.

\section*{Acknowledgements}
The authors would like to thank 
{the editor, the associate editor, and the referees} 
for their comments which helped to improve this paper.  
The authors are also grateful to Judith Rousseau and Eduard Belitser for helpful discussions and comments. 
This work is partially funded by the STAR cluster, ANR Bandhits, VICI Safe Statistics
and Labex ECODEC. This work is part of the second author's PhD.




\bibliography{inv}
\bibliographystyle{apalike}

\end{document}